\documentclass[a4paper]{amsart}
\usepackage{amsmath,amsthm,amssymb,latexsym,epic,bbm,comment}
\usepackage{graphicx,enumerate,stmaryrd}
\usepackage[all]{xy}

\usepackage[parfill]{parskip}

\newtheorem{theorem}{Theorem}
\newtheorem{lemma}[theorem]{Lemma}
\newtheorem{remark}[theorem]{Remark}
\newtheorem{corollary}[theorem]{Corollary}
\newtheorem{proposition}[theorem]{Proposition}
\newtheorem{example}[theorem]{Example}

\newcommand{\tto}{\twoheadrightarrow}

\font\sc=rsfs10
\newcommand{\cC}{\sc\mbox{C}\hspace{1.0pt}}

\newcommand{\cI}{\sc\mbox{I}\hspace{1.0pt}}
\newcommand{\cJ}{\sc\mbox{J}\hspace{1.0pt}}
\newcommand{\cS}{\sc\mbox{S}\hspace{1.0pt}}
\newcommand{\cD}{\sc\mbox{D}\hspace{1.0pt}}

\newcommand{\cA}{\sc\mbox{A}\hspace{1.0pt}}
\newcommand{\cB}{\sc\mbox{B}\hspace{1.0pt}}

\font\scc=rsfs7
\newcommand{\ccC}{\scc\mbox{C}\hspace{1.0pt}}

\newcommand{\ccA}{\scc\mbox{A}\hspace{1.0pt}}
\newcommand{\ccJ}{\scc\mbox{J}\hspace{1.0pt}}

\font\sccc=rsfs5
\newcommand{\cccA}{\sccc\mbox{A}\hspace{1.0pt}}

\begin{document}

\title[Additive versus abelian $2$-representations]{Additive versus 
abelian $2$-representations of fiat $2$-categories}
\author{Volodymyr Mazorchuk and Vanessa Miemietz}
\date{\today}

\begin{abstract}
We study connections between additive and abelian $2$-rep\-re\-sen\-ta\-ti\-ons of fiat $2$-categories, describe 
combinatorics of $2$-categories in terms of multisemigroups and determine the annihilator of a cell $2$-representation. 
We also describe, in detail, examples of fiat $2$-categories associated to $\mathfrak{sl}_2$-categorification in 
the sense of Chuang and Rouquier, and $2$-categorical analogues of Schur algebras.
\end{abstract}

\maketitle

\section{Introduction and description of the results}\label{s0}

$2$-categories appear naturally in the study of ``categorifications'' of various classical mathematical objects,
see for example \cite{Ro0,Ro,MM}. Natural representations of $2$-categories are called $2$-representations and
represent $2$-categories using functorial actions on some classical categories, for example additive or abelian 
categories. In our previous paper \cite{MM} we defined a class of {\em fiat} $2$-categories which we believe is a 
proper $2$-categorical analogue of finite dimensional associative algebras with involution and started the project 
of understanding the $2$-representation theory of such $2$-categories. The main result of \cite{MM} is the construction
and study of what we call {\em cell $2$-representations}, which we believe is a sensible $2$-categorical analogue
of simple modules.

The present paper, on the one hand, continues and extends the study from \cite{MM} and, on the other hand, proposes
an approach to these questions from a slightly different perspective. The paper \cite{MM} mostly studies
$2$-representations of fiat categories in abelian categories. In the present paper we go back to the original point 
of view of Rouquier, see \cite{Ro}, and try to represent fiat categories using additive categories instead. The
advantage is that many constructions are simplified, the disadvantage is that there is not much one can do with
additive $2$-representations. To combine the advantages of both approaches we connect additive and abelian
categorifications via a $2$-functor, called the {\em abelianization} $2$-functor. As an immediate consequence we
obtain natural constructions for many new $2$-representations as follows: given a fiat category $\cC$ and an additive
$2$-representation $\mathbf{M}$ of $\cC$, one can abelianize $\mathbf{M}$ and then find many additive subrepresentations
inside this abelian $2$-representation by taking additive closures of any collection of objects stable (up to 
isomorphism) under the action of $\cC$.

Let us briefly describe the content of the paper. In Section~\ref{s1} we collect all necessary preliminaries
on $2$-categories, $2$-representation and decategorifications together with some classical examples.
Section~\ref{s3} describes the combinatorics of $2$-categories using the language of multisemigroups. A multisemigroup
is a ``semigroup'' with a multivalued operation. Isomorphism classes of $1$-morphisms in a $2$-category form in a
natural way a multisemigroup with respect to the operation induced by the horizontal composition. Similarly to
classical semigroup theory, on each multisemigroup one can define Green's equivalence relations which control
when two elements of a (multi)semigroup generate the same principal left-, right- or two-sided ideals. We explain
that the latter coincide with the notions of left-, right- and two-sided cells for $2$-categories as introduced in \cite{MM}.

In Section~\ref{s2} we discuss the connection between additive and abelian $2$-rep\-re\-sen\-ta\-ti\-ons as described 
above. We define principal additive $2$-representations, the abelianization $2$-functor and explain how to find
additive $2$-subrepresentations in abelian $2$-representations. As an application, we reformulate the construction
of cell $2$-representations of fiat categories from \cite{MM} using this machinery. Section~\ref{s4} is dedicated
to the comparison of various classes of $2$-representations. We show that every abelian representation of a fiat category 
is equivalent to the abelianization of the additive $2$-subrepresentation given by the subcategory of projective objects
and also extend the comparison result for cell $2$-representations from \cite{MM} to the additive setup.

In Section~\ref{s5} we study annihilators of $2$-representations, in particular, those of cell $2$-representations.
In \cite{MM} we show that cell $2$-representations of fiat $2$-categories have properties similar to those
 of simple representations of associative algebras. Here we add another property, formulated below,
which says that the image of a fiat $2$-category on a cell $2$-representation is ``simple'' in the sense that
any nontrivial $2$-ideal necessarily annihilates the cell $2$-representation. Our main result is the following:

\begin{theorem}\label{thmmain}
Let $\mathcal{J}$ be a two-sided cell in a fiat $2$-category $\cC$ and $\mathcal{L}$ be a left cell in $\mathcal{J}$.
Then the ``image'' $2$-category $\cC/\mathrm{Ker}(\mathbf{C}_{\mathcal{L}})$ of the cell 
$2$-representation $\mathbf{C}_{\mathcal{L}}$ is $\mathcal{J}$-simple.
\end{theorem}

Using Theorem~\ref{thmmain}, we propose an alternative construction for cell $2$-rep\-re\-sen\-ta\-ti\-ons by considering
``simple'' quotients of fiat categories associated to two-sided cells.

Finally, in Section~\ref{s6} we describe, in detail, examples of fiat $2$-categories associated to 
$\mathfrak{sl}_2$-categorification in the sense of Chuang and Rouquier (see \cite{CR}), and also  $2$-categorical 
analogues of Schur algebras. We show that both of these examples have particularly nice combinatorial properties
and hence all of our results from \cite{MM} and the present paper are fully applicable. We complete the paper with 
one general construction for fiat categories inspired by the restriction of projective functors to parabolic 
blocks of the BGG category $\mathcal{O}$.

\vspace{5mm}

\noindent
{\bf Acknowledgment.} A substantial part of the paper was written during the visit, supported by ERC grant PERG07-GA-2010-268109, of the first author to the University of East Anglia, whose hospitality is gratefully acknowledged.
The first author is also partially supported by the Swedish Research Council and the Royal Swedish Academy of Sciences.

\section{$2$-categories and $2$-representations}\label{s1}

\subsection{$2$-categories}\label{s1.1}

If not stated otherwise, all categories considered in this paper are assumed to be small.  We denote by $\mathbf{Cat}$
the locally small category of small categories, which is monoidal with respect to the cartesian product. By a 
{\em $2$-category} we mean a category enriched over $\mathbf{Cat}$. In other words, a $2$-category $\cC$ consists
of the following data: a set $\cC$ of objects; for every $\mathtt{i},\mathtt{j}\in\cC$ a small category
$\cC(\mathtt{i},\mathtt{j})$ of morphisms from $\mathtt{i}$ to $\mathtt{j}$ (objects of $\cC(\mathtt{i},\mathtt{j})$
are called {\em $1$-morphisms} of $\cC$ while morphisms of $\cC(\mathtt{i},\mathtt{j})$
are called {\em $2$-morphisms} of $\cC$); the identity $1$-morphism $\mathbbm{1}_{\mathtt{i}}\in 
\cC(\mathtt{i},\mathtt{i})$ for all $\mathtt{i}$; and functorial compositions
$\cC(\mathtt{j},\mathtt{k})\times \cC(\mathtt{i},\mathtt{j})\to \cC(\mathtt{i},\mathtt{k})$; which satisfy the
strict versions of the usual axioms.

We retain the notational conventions from \cite{MM}. For a $2$-category $\cC$, objects of $\cC$ will be denoted by 
$\mathtt{i},\mathtt{j}$ and so on; $1$-morphisms of $\cC$ will be called $\mathrm{F},\mathrm{G}$ and so on;
$2$-morphisms of $\cC$ will be written $\alpha,\beta$ and so on. The identity $2$-endomorphism of a
$1$-morphism $\mathrm{F}$ will be denoted $\mathrm{id}_{\mathrm{F}}$. Composition of $1$-morphisms will be denoted 
by $\circ$, horizontal composition of $2$-morphisms will be denoted by $\circ_0$ and vertical composition of 
$2$-morphisms will be denoted by $\circ_1$. We often abbreviate $\mathrm{id}_{\mathrm{F}}\circ_0\alpha$ and 
$\alpha\circ_0\mathrm{id}_{\mathrm{F}}$ by $\mathrm{F}(\alpha)$ and $\alpha_{\mathrm{F}}$, respectively.
To avoid the cumbersome notation $\cC(\mathtt{i},\mathtt{j})(\mathrm{F},\mathrm{G})$ we will write
$\mathrm{Hom}_{\cC(\mathtt{i},\mathtt{j})}(\mathrm{F},\mathrm{G})$ instead.

\subsection{Fiat categories}\label{s1.2}

Let $\Bbbk$ be a field. A $2$-category $\cC$ is called {\em fiat} (over $\Bbbk$) provided that the following
conditions are satisfied:
\begin{itemize}
\item $\cC$ has finitely many objects;
\item every $\cC(\mathtt{i},\mathtt{j})$ is a fully additive $\Bbbk$-linear category with finitely many 
indecomposable objects up to isomorphism;
\item horizontal compositions are additive and $\Bbbk$-linear;
\item every $\Bbbk$-linear space $\mathrm{Hom}_{\cC(\mathtt{i},\mathtt{j})}(\mathrm{F},\mathrm{G})$ is finite dimensional;
\item all $1$-morphisms $\mathbbm{1}_{\mathtt{i}}$ are indecomposable;
\item $\cC$ has a weak object preserving involutive anti-autoequivalence $*$;
\item $\cC$ has {\em adjunction morphisms}, that is for any $\mathtt{i}, \mathtt{j}\in\cC$ and any $1$-morphism
$\mathrm{F}\in\cC(\mathtt{i},\mathtt{j})$ there exist  $2$-morphisms $\alpha:\mathrm{F}\circ\mathrm{F}^*\to
\mathbbm{1}_{\mathtt{j}}$ and $\beta:\mathbbm{1}_{\mathtt{i}}\to \mathrm{F}^*\circ\mathrm{F}$ such that 
$\alpha_{\mathrm{F}}\circ_1\mathrm{F}(\beta)=\mathrm{id}_{\mathrm{F}}$ and
$\mathrm{F}^*(\alpha)\circ_1\beta_{\mathrm{F}^*}=\mathrm{id}_{\mathrm{F}^*}$.
\end{itemize}

\begin{example}\label{ex1}
{\rm 
Let $A=A_1\oplus A_2\oplus\dots\oplus A_k$ be a finite dimensional associative $\Bbbk$-algebra such that 
each $A_i$ is connected and not simple and the $A_i$ are pairwise non-isomorphic. Assume further that
every $A_i$ is weakly symmetric (i.e. it is self-injective and the top of each indecomposable projective is
isomorphic to the socle). For $i=1,\dots,k$ fix some small category $\mathcal{C}_i$, equivalent to
$A_i\text{-}\mathrm{mod}$. A {\em projective functor} from $\mathcal{C}_i$ to $\mathcal{C}_j$ is a functor 
isomorphic to tensoring with a projective $A_j\text{-}A_i$-bimodule. Let $\cC_{A}$ be the $2$-category with $k$ 
objects $\mathcal{C}_i$, $i=1,\dots,k$; $1$-morphisms being all projective functors between objects together with
all endofunctors of objects isomorphic to the corresponding identity functors; and $2$-morphisms being all
natural transformations of functors. Then $\cC_{A}$ is a fiat category, see \cite[Subsection~7.3]{MM}.
} 
\end{example}

\begin{example}\label{ex2}
{\rm 
Let $\mathfrak{g}$ be a semi-simple complex finite dimensional Lie algebra with a fixed triangular
decomposition $\mathfrak{g}=\mathfrak{n}_-\oplus \mathfrak{h}\oplus \mathfrak{n}_+$, let $W$ be the Weyl
group of $\mathfrak{g}$ and $\mathtt{C}_W$ the coinvariant algebra associated to the action of $W$ on
$\mathfrak{h}^*$. For a simple reflection $s\in W$ let $\mathtt{C}_W^s$ denote the subalgebra of $s$-invariants in
$\mathtt{C}_W$. For $w\in W$ fix some reduced expression $w=s_1\cdot s_2\cdot \cdots\cdot  s_k$ and define the 
$\mathtt{C}_W\text{-}\mathtt{C}_W$-bimodule 
\begin{displaymath}
\hat{B}_w:=\mathtt{C}_W\otimes_{\mathtt{C}_W^{s_1}}\mathtt{C}_W\otimes_{\mathtt{C}_W^{s_2}}\cdots 
\otimes_{\mathtt{C}_W^{s_k}}\mathtt{C}_W. 
\end{displaymath}
Define $B_e:=\mathtt{C}_W$, and for $w\neq e$ define $B_w$ inductively (with respect to the length of $w$) as the 
unique indecomposable direct summand of $\hat{B}_w$ that is not isomorphic to something already defined.
The $\mathtt{C}_W\text{-}\mathtt{C}_W$-bimodule $B_w$ is called {\em Soergel bimodule} associated with $w$ 
(see \cite{So}). Up to isomorphism it does not depend on the choice of  a reduced expression for $w$.
Let $\mathcal{A}$ be some small category equivalent to $\mathtt{C}_W\text{-}\mathrm{mod}$.
Let $\cS_{\mathfrak{g}}$ be the $2$-category with one object $\mathcal{A}$; $1$-morphisms being all endofunctors 
of $\mathcal{A}$ isomorphic to tensoring with direct sums of Soergel bimodules; and $2$-morphisms being all
natural transformations of functors. Then $\cS_{\mathfrak{g}}$ is a fiat category, see \cite[Subsection~7.1]{MM},
the so-called {\em $2$-category of Soergel bimodules}.
} 
\end{example}

From now on, if not explicitly stated otherwise, $\cC$ is always assumed to be a fiat category.

\subsection{$2$-representations}\label{s1.3}

Denote by $\mathfrak{A}_{\Bbbk}$ the locally small $2$-subcategory of $\mathbf{Cat}$ defined as follows: objects of 
$\mathfrak{A}_{\Bbbk}$ are small fully additive $\Bbbk$-linear categories, $1$-morphisms are  additive $\Bbbk$-linear
functors and $2$-morphisms are natural transformations of functors. Denote by $\mathfrak{A}_{\Bbbk}^{f}$ the full 
subcategory of $\mathfrak{A}_{\Bbbk}$ containing all small fully additive $\Bbbk$-linear categories which have a finite 
number of indecomposable objects up to isomorphism. Finally, denote by $\mathfrak{R}_{\Bbbk}$ the full subcategory of 
$\mathfrak{A}_{\Bbbk}$ containing all small fully additive $\Bbbk$-linear categories that are equivalent to 
module categories of finite dimensional associative $\Bbbk$-algebras.

Given a $2$-category $\cC$, a {\em $2$-representation} of $\cC$ is a $2$-functor from $\cC$ to one of the $2$-categories
$\mathfrak{A}_{\Bbbk}$, $\mathfrak{A}_{\Bbbk}^f$ or $\mathfrak{R}_{\Bbbk}$. Representations of $\cC$ in
$\mathfrak{A}_{\Bbbk}$, $\mathfrak{A}_{\Bbbk}^f$ or $\mathfrak{R}_{\Bbbk}$ are called {\em additive}, {\em finitary} 
and {\em abelian} (or of additive, finitary or abelian {\em type}), respectively. Together with 
$2$-natural transformations and modifications, additive, finitary
and abelian representations of $\cC$ form $2$-categories $\cC\text{-}\mathrm{amod}$, $\cC\text{-}\mathrm{afmod}$
and $\cC\text{-}\mathrm{mod}$, respectively. We refer the reader to \cite{McL,Ma,MM} for more details on
$2$-functors, $2$-natural transformations and modifications.

\begin{example}\label{ex3}
{\rm  
Both the category $\cC_{A}$ from Example~\ref{ex1} and the category $\cS_{\mathfrak{g}}$ from Example~\ref{ex2}
are defined via the corresponding {\em defining representation} (given by the identity maps on all ingredients).
}
\end{example}

\subsection{Decategorification}\label{s1.4}

The {\em decategorification} of a fiat category $\cC$ is a category (i.e. a $1$-category) $[\cC]$ defined as follows: 
$[\cC]$ has the same objects as $\cC$; for $\mathtt{i},\mathtt{j}\in [\cC]$ the morphism space 
$[\cC](\mathtt{i},\mathtt{j})$ is defined as the split Grothendieck group $[\cC(\mathtt{i},\mathtt{j})]$ of the 
additive category $\cC(\mathtt{i},\mathtt{j})$; the composition in $[\cC]$ is induced by the composition in $\cC$
(note that the latter is biadditive). 

Given an additive $2$-representations $\mathbf{M}$ of $\cC$, the {\em decategorification} of $\mathbf{M}$
is a functor $[\mathbf{M}]$ from $[\cC]$ to the category $\mathbf{Ab}$ of abelian groups defined as follows:
for $\mathtt{i}\in[\cC]$ the abelian group $[\mathbf{M}](\mathtt{i})$ is defined as the split Grothendieck
group of $\mathbf{M}(\mathtt{i})$ and for any $1$-morphism $\mathrm{F}\in\cC(\mathtt{i},\mathtt{j})$, the
action of the class $[\mathrm{F}]\in[\cC(\mathtt{i},\mathtt{j})]$ on $[\mathbf{M}](\mathtt{i})$ is induced
by the functorial action of $\mathrm{F}$ on $\mathbf{M}(\mathtt{i})$ (note that the latter functor is additive).
In this way, the decategorification of a $2$-representation of $\cC$ is a representation of $[\cC]$.
This gives rise to a functor from the category of additive $2$-representations of $\cC$ (with modifications forgotten) 
to the category of representations of $[\cC]$ in $\mathbf{Ab}$.

As $\cC$ is fiat, every $1$-morphism of $\cC$ acts as a exact functor on any abelian $2$-representation of $\cC$.
Hence one similarly defines the decategorification of an abelian $2$-representation of $\cC$ using the
usual Grothendieck group of an abelian category.

\subsection{Equivalent $2$-representations}\label{s1.5}

Two $2$-representations $\mathbf{M}$ and $\mathbf{N}$ of $\cC$ are called {\em elementary equivalent} provided that
there exists either a $2$-natural transformation $\Phi:\mathbf{M}\to \mathbf{N}$ or
a $2$-natural transformation $\Phi:\mathbf{N}\to \mathbf{M}$ such that $\Phi_{\mathtt{i}}$
is an equivalence for any $\mathtt{i}\in\cC$.  Two $2$-representations $\mathbf{M}$ and $\mathbf{N}$ of $\cC$ 
are called {\em equivalent} provided that there is a sequence $\mathbf{M}_1,\mathbf{M}_2,\dots,\mathbf{M}_k$
of $2$-representations of $\cC$ such that  $\mathbf{M}_1=\mathbf{M}$, $\mathbf{M}_k=\mathbf{N}$ and
for any $i=1,2,\dots, k-1$ the $2$-representations $\mathbf{M}_i$ and $\mathbf{M}_{i+1}$ are
elementary equivalent. Note that equivalent $2$-representations of $\cC$ descend
to isomorphic representations of $[\cC]$.

Two $2$-representations $\mathbf{M}$ and $\mathbf{N}$ of $\cC$ are called {\em isomorphic} provided that
there exists either a $2$-natural transformation $\Phi:\mathbf{M}\to \mathbf{N}$ such that $\Phi_{\mathtt{i}}$
is an isomorphism for any $\mathtt{i}\in\cC$. In this case the inverse isomorphism $\Phi^{-1}:\mathbf{N}\to \mathbf{M}$
is a $2$-natural transformation as well.

\section{Combinatorics of additive $2$-categories}\label{s3}

Before coming to $2$-representations, we first address the question of internal structure of $2$-categories.
In this section we describe the multisemigroup approach to the combinatorics of additive $\Bbbk$-linear 
$2$-categories. This is inspired by \cite{Vi}.

\subsection{Multisemigroups}\label{s3.1}

For a set $X$ we will denote by $\mathcal{B}(X)$ the set of all subsets of $X$ (the Boolean of $X$). For
$x\in X$ we will identify the element $x$ with the subset $\{x\}\in \mathcal{B}(X)$.

A {\em multisemigroup} is a pair $(S,*)$ consisting of a set $S$ and a multivalued operation 
$*:S\times S\to \mathcal{B}(S)$ (written $(a,b)\mapsto a*b$), which is assumed to be {\em associative} in the 
sense that for all $a,b,c\in S$ one has
\begin{displaymath}
\bigcup_{t\in a*b} t*c = \bigcup_{s\in b*c} a*s.
\end{displaymath}
For example, any semigroup is a multisemigroup. Many other examples of multisemigroup can be found in \cite{Vi},
we just give one more in Example~\ref{ex5} below.

A {\em unit} of a multisemigroup $(S,*)$ is an element $e\in S$ such that $a*e=e*a=a$. If a multisemigroup has
a unit, this unit is unique. If a multisemigroup $(S,*)$ does not have a unit, one can formally adjoin it by
considering the multisemigroup $(S^1,\circledast)$, where $S^1=S\cup\{1\}$ for some $1\not\in S$, and defining
the operation $\circledast$ for $a,b\in S^1$ as follows:
\begin{displaymath}
a \circledast b=\begin{cases}a,&b=1;\\b,&a=1;\\a*b,&\text{otherwise}.\end{cases}
\end{displaymath}

\begin{example}\label{ex5}
{\em 
Consider the set $\mathbb{Z}_+=\{0,1,2,3,\dots\}$ of all non-negative integers and for $m,n\in \mathbb{Z}_+$ 
define 
\begin{displaymath}
m\diamond n=\{x\in \mathbb{Z}_+: |m-n|\leq x\leq m+n,\,\, x\equiv m+n\,\mathrm{mod}\, 2\}.
\end{displaymath}
This makes $(\mathbb{Z}_+,\diamond )$ into a multisemigroup. The element $0$ is the unit 
element of $(\mathbb{Z}_+,\diamond)$.
} 
\end{example}

Any multisemigroup $(S,*)$ induces a semigroup structure on $\mathcal{B}(S)$ by extending the operation $*$ to the
whole of $\mathcal{B}(S)$ as follows (here $X,Y\in \mathcal{B}(S)$):
\begin{displaymath}
X*Y =\bigcup_{x\in X,\,y\in Y} x*y.
\end{displaymath}
In what follows we will often use this extension of $*$.

\subsection{Green's relations in multisemigroups}\label{s3.2}

Here we define analogues of the classical Green's relations for semigroups (see \cite{Gr}) in our more general
setup of multisemigroups. Let $(S,*)$ be a multisemigroup. For $a,b\in S$ we will write $a\leq_L b$, 
$a\leq_R b$ or $a\leq_J b$ in the cases $S^1*b\subset S^1*a$, $b*S^1\subset a*S^1$ and $S^1*b*S^1\subset S^1*a*S^1$, 
respectively. The partial pre-order relations $\leq_L$, $\leq_R$ and $\leq_J$ will be called the {\em left}, 
{\em right} and {\em two-sided} {\em orders}, respectively. Equivalence classes 
$\sim_L$, $\sim_R$ and $\sim_J$ of these relations will be called 
the {\em left}, {\em right} and  {\em two-sided} {\em cells}, respectively. If $S$ is a semigroup, then left, right
and two-sided cells coincide with Green's $\mathcal{L}$-, $\mathcal{R}$- and $\mathcal{J}$-classes, respectively. 

\begin{example}\label{ex6}
{\em 
For the multisemigroup $(\mathbb{Z}_+,\diamond)$ in Example~\ref{ex5} the whole multisemigroup is the unique cell,
which is both left, right and two-sided at the same time.
} 
\end{example}

\subsection{Multisemigroup of a $2$-category}\label{s3.3}

For this subsection we just assume that $\cC$ is small and fully additive (that is $\cC(\mathtt{i},\mathtt{j})$ 
is additive for all $\mathtt{i}$ and $\mathtt{j}$) and that the composition bifunctor is biadditive. Denote by 
$\mathcal{S}(\cC)$ the set of isomorphism classes of indecomposable $1$-morphisms in $\cC$ with an element $0$ added. 
For indecomposable $1$-morphisms $\mathrm{F}$ and $\mathrm{G}$ let $[\mathrm{F}]$ and $[\mathrm{G}]$ denote 
the corresponding classes in $\mathcal{S}(\cC)$ and set
\begin{displaymath}
[\mathrm{F}]*[\mathrm{G}]=
\begin{cases}
0, & \mathrm{F}\circ \mathrm{G} \text{ is undefined};\\
0, & \mathrm{F}\circ \mathrm{G}=0;\\
\{[\mathrm{H}]\in\mathcal{S}(\cC) : \mathrm{H}\text{ is a summand of }\mathrm{F}\circ \mathrm{G}\}, & \text{ otherwise}.
\end{cases}
\end{displaymath}
Associativity of horizontal composition in $\cC$ implies that $\mathcal{S}(\cC)$ becomes a multisemigroup,
which we will call the {\em multisemigroup} of $\cC$. In case $[\mathrm{F}]*[\mathrm{G}]\neq 0$ for all $[\mathrm{F}]$
and $[\mathrm{G}]$ one can consider the multisubsemigroup $\mathcal{S}'(\cC):=\mathcal{S}(\cC)\setminus\{0\}$ 
of $\mathcal{S}'(\cC)$, which we will call the {\em reduced multisemigroup} of $\cC$.

\begin{example}\label{ex7}
{\em 
Let $\cC_{\mathfrak{sl}_2}$ denote the $2$-category with one object $\mathtt{i}$ such that 
$\cC_{\mathfrak{sl}_2}(\mathtt{i},\mathtt{i})$ is a small category equivalent to the category of finite dimensional
$\mathfrak{sl}_2(\mathbb{C})$-modules and such that the composition in $\cC_{\mathfrak{sl}_2}$ is induced by the 
tensor product of $\mathfrak{sl}_2(\mathbb{C})$-modules. Mapping an $n+1$-dimensional simple
$\mathfrak{sl}_2(\mathbb{C})$-module to $n$ identifies $\mathcal{S}'(\cC)$ with $\mathbb{Z}_+$. 
The classical Clebsh-Gordon rule for $\mathfrak{sl}_2(\mathbb{C})$ implies that $\mathcal{S}'(\cC)$ is
isomorphic to the multisemigroup $(\mathbb{Z}_+,\diamond)$ from Example~\ref{ex5}.
} 
\end{example}

\begin{remark}\label{rem8}
{\rm
One can also consider the natural generalization of multisemigroups using multisets (such objects perhaps should be 
then called {\em multimultisemigroups}). Similarly to the above one then defines the multimultisemigroup 
$\mathcal{MS}(\cC)$ for a small and fully additive $2$-category $\cC$. This multimultisemigroup then contains
complete information about the combinatorics of composition in $\cC$, including all multiplicities in the direct sum
decomposition of the composition of two $1$-morphisms. The multisemigroup $\mathcal{S}(\cC)$ contains only a 
``shadow'' of this information, saying only which direct summands appear in the direct sum decomposition of the 
composition of two $1$-morphisms.
} 
\end{remark}

\section{Additive and abelian $2$-representations}\label{s2}

Let $\cC$ be a fiat category. In \cite{MM} we worked with abelian $2$-representations of $\cC$. In the present paper we
would like to slightly change the approach and start with additive $2$-representations (as was originally suggested
in \cite{Ro}). We also formulate a direct $2$-functorial link between these two classes of representations.

\subsection{Principal additive $2$-representations}\label{s2.1}

For $\mathtt{i}\in\cC$ let $\mathbb{P}_{\mathtt{i}}:\cC\to \mathfrak{A}_{\Bbbk}$ denote the $\mathtt{i}$-th
{\em principal additive} $2$-representation $\cC(\mathtt{i},{}_-)$. This $2$-representation assigns to each 
$\mathtt{j}\in\cC$ the category $\cC(\mathtt{i},\mathtt{j})$, to each $1$-morphism $\mathrm{F}$ the functor given by 
the horizontal composition with $\mathrm{F}$, and to each $2$-morphism $\alpha$ the natural transformation given by 
the horizontal composition with $\alpha$. This $2$-representation is finitary as $\cC$ is fiat.

\begin{lemma}[Yoneda lemma]\label{lem4}
For any $\mathbf{M}\in \cC\text{-}\mathrm{amod}$ we have an isomorphism
\begin{displaymath}
\mathrm{Hom}_{\cC\text{-}\mathrm{amod}} (\mathbb{P}_{\mathtt{i}},\mathbf{M})=\mathbf{M}(\mathtt{i}).
\end{displaymath}
\end{lemma}

\begin{proof}
Given $M\in \mathbf{M}(\mathtt{i})$ there is a unique homomorphism $\Phi_M:\mathbb{P}_{\mathtt{i}}\to \mathbf{M}$ such 
that $\Phi_M(\mathbbm{1}_{\mathtt{i}})=M$, namely the one sending $\mathrm{F}\in \cC(\mathtt{i},\mathtt{j})$ to
$\mathrm{F}\, M$ and $\alpha:\mathrm{F}\to \mathrm{F}'$ to $\alpha_M$.

Given $M,N\in \mathbf{M}(\mathtt{i})$ and $\tau\in\mathrm{Hom}_{\mathbf{M}(\mathtt{i})}(M,N)$, mapping 
$\mathrm{F}\in \cC(\mathtt{i},\mathtt{j})$ to $\mathrm{F}(\tau)$ defines a modification from $\Phi_M$ to $\Phi_N$. 
On the other hand, given a modification $\xi:\Phi_M\to \Phi_N$, let $\tau=\xi(\mathbbm{1}_{\mathtt{i}})$. Then
for any $\mathrm{F}\in\cC(\mathtt{i},\mathtt{j})$ the axioms of modification imply 
\begin{displaymath}
\mathrm{F}(\tau)= \mathrm{F}(\xi(\mathbbm{1}_{\mathtt{i}}))=
\xi(\mathrm{F}\,\mathbbm{1}_{\mathtt{i}})=\xi(\mathrm{F})
\end{displaymath}
and hence $\xi$ is uniquely determined by $\tau$. The claim follows.
\end{proof}

\subsection{The abelianization $2$-functor}\label{s2.2}

Given $\mathbf{M}\in\cC\text{-}\mathrm{afmod}$ we define a new abelian $2$-representation $\overline{\mathbf{M}}$
of $\cC$ in the following way. For $\mathtt{i}\in\cC$ define the category $\overline{\mathbf{M}}(\mathtt{i})$ as
follows: objects of $\overline{\mathbf{M}}(\mathtt{i})$ are all diagrams of the form $\xymatrix{X\ar[r]^{\alpha}&Y}$,
where $X,Y\in \mathbf{M}(\mathtt{i})$ and $\alpha\in\mathrm{Hom}_{\mathbf{M}(\mathtt{i})}(X,Y)$. The homomorphism
set 
\begin{displaymath}
\mathrm{Hom}_{\overline{\mathbf{M}}(\mathtt{i})}(\xymatrix{X\ar[r]^{\alpha}&Y},\xymatrix{X'\ar[r]^{\alpha'}&Y'})
\end{displaymath}
is defined as the quotient of the linear space generated by all commutative diagrams as shown in the solid part of the
following picture:
\begin{displaymath}
\xymatrix{
X\ar[r]^{\alpha}\ar[d]_{\beta}&Y\ar[d]^{\gamma}\ar@{.>}[dl]_{\xi}\\
X'\ar[r]^{\alpha'}&Y' 
} 
\end{displaymath}
modulo the subspace spanned by all diagrams for which there exists $\xi$ as shown on the picture such that 
$\gamma=\alpha'\xi$. The $2$-action of $\cC$ is defined on such diagrams in the natural way, that is component-wise,
which makes $\overline{\mathbf{M}}$ into an abelian $2$-representation of $\cC$. Extending $2$-natural transformations
and modifications in $\cC\text{-}\mathrm{afmod}$ to diagrams component-wise defines a $2$-functor from
$\cC\text{-}\mathrm{afmod}$ to $\cC\text{-}\mathrm{mod}$ which we will call the {\em abelianization} functor
and denote by $\overline{\,\,\cdot\,\,}:\cC\text{-}\mathrm{afmod}\to\cC\text{-}\mathrm{mod}$.
Applying this $2$-functor to $\mathbb{P}_{\mathtt{i}}$ yields the {\em principal abelian} $2$-representation
$\overline{\mathbb{P}}_{\mathtt{i}}$ considered in \cite{MM}.

\begin{remark}\label{rm73}
{\rm
Let $\mathfrak{T}_{\Bbbk}$ denote the locally small $2$-subcategory of $\mathbf{Cat}$ with objects 
small triangulated $\Bbbk$-linear categories, $1$-morphisms triangle functors and $2$-morphisms
natural transformations of functors. Denote by $\cC\text{-}\mathrm{tmod}$ the category of 
$2$-representations of $\cC$ in $\mathfrak{T}_{\Bbbk}$. Similarly to the above, using the usual 
construction of the homotopy category of an additive category, one defines a $2$-functor 
from $\cC\text{-}\mathrm{afmod}$ to $\cC\text{-}\mathrm{tmod}$.
}
\end{remark}

\subsection{Additive subrepresentations of abelian representations}\label{s2.3}

Assume that we are given $\mathbf{M}\in\cC\text{-}\mathrm{mod}$, $\mathtt{i}\in\cC$ and $X\in\mathbf{M}(\mathtt{i})$.
For $\mathtt{j}\in\cC$ define $\mathbf{M}_{X}(\mathtt{j})$ as $\mathrm{add}(\mathrm{F}\,X)$, where $\mathrm{F}$ runs
through the set of all $1$-morphisms in $\cC(\mathtt{i},\mathtt{j})$. Since the number of $1$-morphisms in 
$\cC(\mathtt{i},\mathtt{j})$ is finite up to isomorphism, we have $\mathbf{M}_{X}(\mathtt{j})\in\mathfrak{A}_{\Bbbk}^f$.
Via restriction from $\mathbf{M}$, $\mathbf{M}_{X}$ becomes a $2$-representation of $\cC$ and from the previous
observation we obtain that $\mathbf{M}_{X}\in \cC\text{-}\mathrm{afmod}$. This is a very general and powerful tool for 
constructing new $2$-representations.

\subsection{Cell $2$-representation}\label{s2.4}

In this subsection we recall the construction of a special class of $2$-representations defined and studied in \cite{MM}.
Let $\mathcal{L}$ be a (nonzero) left cell of $\mathcal{S}(\cC)$. To simplify notation, we will from now on identify 
indecomposable $1$-morphisms in $\cC$ with the corresponding classes in $\mathcal{S}(\cC)$, in particular, we will
write $\mathrm{F}\in \mathcal{S}(\cC)$ instead of $[\mathrm{F}]\in \mathcal{S}(\cC)$. 

There is $\mathtt{i}=\mathtt{i}_{\mathcal{L}}\in\cC$ such that for any $1$-morphism $\mathrm{F}\in \mathcal{L}$ 
we have $\mathrm{F}\in \cC(\mathtt{i},\mathtt{j})$ for some $\mathtt{j}\in\cC$. Consider 
$\overline{\mathbb{P}}_{\mathtt{i}}$. For an indecomposable $1$-morphism $\mathrm{F}$ in some 
$\cC(\mathtt{i},\mathtt{j})$ denote by $P_{\mathrm{F}}$ the indecomposable projective module $0\to \mathrm{F}$
in $\overline{\mathbb{P}}_{\mathtt{i}}(\mathtt{j})$ and by $L_{\mathrm{F}}$ the unique simple top of $P_{\mathrm{F}}$. 
By \cite[Proposition~17]{MM}, there exists the unique $\mathrm{G}_{\mathcal{L}}\in \mathcal{L}$ (called the
{\em Duflo involution} in $\mathcal{L}$) such that the indecomposable projective module $P_{\mathbbm{1}_{\mathtt{i}}}$
has a unique quotient $N$ such that the simple socle of $N$ is isomorphic to $L_{\mathrm{G}_{\mathcal{L}}}$ and
$\mathrm{F}\, N/L_{\mathrm{G}_{\mathcal{L}}}=0$ for any $\mathrm{F}\in \mathcal{L}$. Set
$Q:=\mathrm{G}_{\mathcal{L}} \, L_{\mathrm{G}_{\mathcal{L}}}$. Then the additive $2$-representation
$\mathbf{C}_{\mathcal{L}}:=\left(\overline{\mathbb{P}}_{\mathtt{i}}\right)_{Q}$ is called the {\em additive cell} 
$2$-representation of $\cC$ associated to $\mathcal{L}$. The abelianization $\overline{\mathbf{C}}_{\mathcal{L}}$
of $\mathbf{C}_{\mathcal{L}}$ is called the {\em abelian cell} $2$-rep\-re\-sen\-ta\-ti\-on of $\cC$ associated 
to $\mathcal{L}$.

\section{Equivalent $2$-representations}\label{s4}

\subsection{Abelian representations are defined by the action on projectives}\label{s4.1}

Let $\mathbf{M}\in\cC\text{-}\mathrm{mod}$. As $\cC$ is fiat, the action of any $1$-morphism $\mathrm{F}$ in $\cC$
on $\mathbf{M}$ is given by a functor biadjoint to an exact functor (representing the action of $\mathrm{F}^*$).
In particular, it follows that any such $\mathrm{F}$ sends projective objects to projective objects. For every
$\mathtt{i}$ denote by $\mathbf{M}_{pr}(\mathtt{i})$ the additive subcategory of $\mathbf{M}(\mathtt{i})$ consisting
of all projective objects. Then the restriction of the action of $\cC$ defines the additive $2$-representation
$\mathbf{M}_{pr}\in\cC\text{-}\mathrm{afmod}$. The main result of this subsection is the following:

\begin{theorem}\label{thm9}
The $2$-representations $\mathbf{M}$ and $\overline{\mathbf{M}_{pr}}$ are equivalent. 
\end{theorem}

\begin{proof}
We prove the claim by constructing a sequence of elementary equivalent $2$-representations of $\cC$ connecting 
$\mathbf{M}$ with $\overline{\mathbf{M}_{pr}}$.

Define the $2$-representation $\mathbf{N}$ of $\cC$ as follows: for $\mathtt{i}\in\cC$ the category
$\mathbf{N}(\mathtt{i})$ has objects $(\mathrm{F}\, X)_{(\mathrm{F})}$, where $X\in \mathbf{M}(\mathtt{j})$ and
$\mathrm{F}$ is a $1$-morphism in $\cC(\mathtt{j},\mathtt{i})$ for some $\mathtt{j}$. 
For $(\mathrm{F}\, X)_{(\mathrm{F})},(\mathrm{F}'\, X')_{(\mathrm{F}')}\in
\mathbf{N}(\mathtt{i})$ set 
\begin{equation}\label{eq1}
\mathrm{Hom}_{\mathbf{N}(\mathtt{i})}((\mathrm{F}\, X)_{(\mathrm{F})},(\mathrm{F}'\, X')_{(\mathrm{F}')}):=
\mathrm{Hom}_{\mathbf{M}(\mathtt{i})}(\mathrm{F}\, X,\mathrm{F}'\, X').
\end{equation}
Define the action of $1$-morphisms of $\cC$ on objects by 
$\mathrm{G}\, (\mathrm{F}\, X)_{(\mathrm{F})}:=(\mathrm{G}\circ \mathrm{F}\, X)_{(\mathrm{G}\circ \mathrm{F})}$. 
Define the action of $1$-morphisms of $\cC$ on morphisms by inducing the action from the one on 
$\mathbf{M}(\mathtt{i})$ in the natural way. Define the action of $2$-morphisms similarly by inducing it from 
$\mathbf{M}(\mathtt{i})$ in the  natural way. This gives an abelian $2$-representation of $\cC$. Sending 
$(\mathrm{F}\, X)_{(\mathrm{F})}\in \mathbf{N}(\mathtt{i})$ to  $\mathrm{F}\, X\in \mathbf{M}(\mathtt{i})$ and using 
the identity map on morphisms we get the forgetful $2$-natural transformation  from $\mathbf{N}$ to $\mathbf{M}$ which 
obviously restricts to an equivalence for every $\mathtt{i}$.  Hence $\mathbf{M}$ and $\mathbf{N}$ are elementary  
equivalent.

For any $\mathtt{i}\in\cC$ and any object $X\in \mathbf{M}(\mathtt{i})$ fix some projective presentation
$P_X\overset{\alpha}{\to} Q_X\tto X$ of $X$ in $\mathbf{M}(\mathtt{i})$. The assignment 
\begin{displaymath}
(\mathrm{F}\, X)_{(\mathrm{F})}\mapsto
(\mathrm{F}\,P_X\overset{\mathrm{F}\, \alpha}{\longrightarrow} \mathrm{F}\,Q_X)
\end{displaymath}
extends to a $2$-natural transformation $\mathbf{N}\to \overline{\mathbf{M}_{pr}}$ in the obvious way. 
Clearly, the restriction of this $2$-natural transformation to any $\mathtt{i}$ is an equivalence. Hence 
$\mathbf{N}$ and  $\overline{\mathbf{M}_{pr}}$ are elementary equivalent. The claim of the theorem follows. 
\end{proof}

\subsection{Comparison of cell $2$-representations}\label{s4.2}

In this subsection we present the additive adaptation of \cite[Theorem~1]{MM}. First we note that
our definition of an abelian cell $2$-representation is slightly different from the one in \cite{MM} (here we define
the abelian cell $2$-representation as abelianization of a certain additive $2$-subrepresentation 
$\mathbf{N}$ of $\overline{\mathbb{P}}_{\mathtt{i}}$, while in \cite{MM} the abelian cell $2$-representation was 
defined as an abelian $2$-subrepresentation $\mathbf{M}$ of $\overline{\mathbb{P}}_{\mathtt{i}}$ such that 
$\mathbf{N}=\mathbf{M}_{pr}$). By Theorem~\ref{thm9}, these two different definitions 
produce equivalent  $2$-representations. We also note that we slightly change the notation from \cite{MM}
because of the introduction of the abelianization functor (and thus $\mathbf{C}_{\mathcal{L}}$ is an additive
$2$-representation in this paper while it was abelian in \cite{MM}). We think our present notation is more natural.

Let $\cC$ be a fiat category, $\mathcal{J}$ a two-sided cell of $\cC$ and $\mathcal{L}$ a left cell in  $\mathcal{J}$. 
The cell $\mathcal{J}$ is called {\em strongly regular} (see \cite[4.8]{MM})
provided that any two different left cells in $\mathcal{J}$ are not comparable with respect to $\leq_L$, and
the intersection of each left and each right cell in $\mathcal{J}$ consists of exactly one element. For example,
all two-sided cells for the fiat category $\cC_A$ in Example~\ref{ex1} are strongly regular. Similarly, all
two-sided cells for the fiat category $\cS_{\mathfrak{sl}_n}$ in Example~\ref{ex2} are strongly regular
(however, this is not the case if $\mathfrak{g}$ is not of type $A$).
Assume that $\mathcal{J}$ is strongly regular. Then for any $\mathrm{F}\in \mathcal{J}$ the intersection of the
left cell of $\mathrm{F}$ with the right cell of $\mathrm{F}^*$ consists of a unique element, say $\mathrm{H}$.
Let $m_{\mathrm{F}}$ denote the multiplicity of $\mathrm{H}$ in $\mathrm{F}^*\circ \mathrm{F}$.

\begin{proposition}\label{prop10}
Assume that $\mathcal{J}$ is strongly regular and that 
\begin{equation}\label{eq2}
\text{the function }\, \mathrm{F}\mapsto m_{\mathrm{F}}\, \text{ is constant
on right cells of }\, \mathcal{J}. 
\end{equation}
Then for any two left cells $\mathcal{L}$ and $\mathcal{L}'$ in $\mathcal{J}$
the $2$-representations $\mathbf{C}_{\mathcal{L}}$ and $\mathbf{C}_{\mathcal{L}'}$ are equivalent
and, similarly, their abelianizations are equivalent as well.
\end{proposition}

Note that in \cite{MM} it was shown that the technical condition of Proposition~\ref{prop10} is satisfied for
strongly regular cells in Examples~\ref{ex1} and \ref{ex2}.

\begin{proof}
Let $\mathrm{H}$ be the unique element in the intersection of $\mathcal{L}'$ with the right cell of 
$\mathrm{G}_{\mathcal{L}}$. Assume $\mathrm{G}_{\mathcal{L}}\in\cC(\mathtt{i},\mathtt{i})$ and
$\mathrm{G}_{\mathcal{L}'}\in\cC(\mathtt{j},\mathtt{j})$. By Lemma~\ref{lem4} we have a unique homomorphism 
$\Phi:\mathbb{P}_{\mathtt{i}}\to \overline{\mathbb{P}}_{\mathtt{j}}$ mapping $\mathbbm{1}_{\mathtt{i}}$
to $L_{\mathrm{H}}$. 

Let $\mathbf{N}$ denote the additive $2$-subrepresentation of $\mathbb{P}_{\mathtt{i}}$ obtained by restriction
to the full additive subcategory generated by all $1$-morphisms $\mathrm{F}$ satisfying 
$\mathrm{F}\geq_L\mathrm{G}_{\mathcal{L}}$. Let $\mathbf{M}$ denote the additive $2$-subrepresentation of $\overline{\mathbb{P}}_{\mathtt{j}}$ obtained by restriction to the full additive subcategory generated 
by $\mathrm{F}\, L_{\mathrm{H}}$ for all $1$-morphisms $\mathrm{F}$ satisfying 
$\mathrm{F}\geq_L\mathrm{G}_{\mathcal{L}}$. From \cite[Theorem~43(a)]{MM} it follows that the latter additive 
subcategory coincides with the additive closure of $\mathrm{F}\, L_{\mathrm{H}}$, $\mathrm{F}\in\mathcal{L}'$,
and hence with $\mathbf{C}_{\mathcal{L}'}$. Thus $\Phi$ induces, by restriction, a homomorphism 
$\Phi':\mathbf{N}\to \mathbf{C}_{\mathcal{L}'}$.

Let $\mathrm{F}\overset{\alpha}{\longrightarrow}\mathrm{G}_{\mathcal{L}}$ be the 
object of $\overline{\mathbf{N}}(\mathtt{i})$ obtained from $L_{\mathrm{G}_{\mathcal{L}}}$ by deleting all summands 
outside $\mathbf{N}(\mathtt{i})$. From \cite[Theorem~43(a)]{MM} it follows that the abelianization of $\Phi'$ sends 
$\mathrm{F}\overset{\alpha}{\longrightarrow}\mathrm{G}_{\mathcal{L}}$ to a simple 
object in $\overline{\mathbf{C}}_{\mathcal{L}'}(\mathtt{i})$ isomorphic to the top of the indecomposable 
projective object $0\to \mathrm{G}_{\mathcal{L}}\,L_{\mathrm{H}}$. This 
induces an equivalence between $\mathbf{C}_{\mathcal{L}}$ and $(\overline{\mathbf{C}}_{\mathcal{L}'})_{pr}$. 
Mapping $P$ to $0\to P$ for any object $P$ induces an obvious equivalence from $\mathbf{C}_{\mathcal{L}'}$  
to $(\overline{\mathbf{C}}_{\mathcal{L}'})_{pr}$ and the additive claim follows. The abelian claim 
follows from Theorem~\ref{thm9} applying the abelianization functor.
\end{proof}

\section{Annihilators of cell $2$-representations}\label{s5}

\subsection{$2$-ideals and quotients}\label{s5.1}

A {\em $2$-ideal} $\cI$ of $\cC$ consists of the following data:
\begin{itemize}
\item the same objects as $\cC$;
\item the same $1$-morphisms as $\cC$;
\item for any $\mathtt{i},\mathtt{j}\in\cC$ and any $1$-morphisms 
$\mathrm{F},\mathrm{G}\in\cC(\mathtt{i},\mathtt{j})$ a $\Bbbk$-subspace 
$\mathrm{Hom}_{\cI}(\mathrm{F},\mathrm{G})\subset\mathrm{Hom}_{\cC(\mathtt{i},\mathtt{j})}(\mathrm{F},\mathrm{G})$
of $2$-morphisms;
\end{itemize}
such that for any $2$-morphisms $\xi$ in $\cI$ and $\alpha,\beta$ in $\cC$ the compositions
$\alpha\circ_0\xi\circ_0\beta$ and $\alpha\circ_1\xi\circ_1\beta$ are in $\cI$ whenever the expression makes sense.

If $\cI$ is a $2$-ideal of $\cC$, then we can define the quotient $2$-category $\cC/\cI$ as follows:
$\cC/\cI$ has the same objects and $1$-morphisms as $\cC$, and for for any $\mathtt{i},\mathtt{j}\in\cC$ and any 
$1$-morphisms $\mathrm{F},\mathrm{G}\in\cC(\mathtt{i},\mathtt{j})$ we set
\begin{displaymath}
\mathrm{Hom}_{\cC/\cI}(\mathrm{F},\mathrm{G}):=
\mathrm{Hom}_{\cC(\mathtt{i},\mathtt{j})}(\mathrm{F},\mathrm{G})/\mathrm{Hom}_{\cI}(\mathrm{F},\mathrm{G});
\end{displaymath}
all compositions in $\cC/\cI$ are induced by the corresponding compositions in $\cC$. The fact that this is 
well-defined follows directly from the axioms of a $2$-ideal.

If $\Phi:\cC\to \cA$ is a $2$-functor between two fiat categories, define the {\em kernel} of $\mathrm{Ker}(\Phi)$ 
as the datum consisting of the same objects and $1$-morphisms as $\cC$, and for for any $\mathtt{i},\mathtt{j}\in\cC$ 
and any $1$-morphisms $\mathrm{F},\mathrm{G}\in\cC(\mathtt{i},\mathtt{j})$ the subspace
$\mathrm{Hom}_{\mathrm{Ker}(\Phi)}(\mathrm{F},\mathrm{G})$ which coincides with the kernel of the $\Bbbk$-linear map
\begin{displaymath}
\Phi_{\mathrm{F},\mathrm{G}}:\mathrm{Hom}_{\cC(\mathtt{i},\mathtt{j})}(\mathrm{F},\mathrm{G})\to
\mathrm{Hom}_{\cA(\Phi(\mathtt{i}),\Phi(\mathtt{j}))}(\Phi(\mathrm{F}),\Phi(\mathrm{G}))
\end{displaymath}
given by the application of $\Phi$. We have the following usual property that $2$-ideals are exactly kernels of 
$2$-functors.

\begin{lemma}\label{lem11}
\begin{enumerate}[(a)]
\item\label{lem11.1} $\mathrm{Ker}(\Phi)$ is a $2$-ideal of $\cC$.
\item\label{lem11.2} For any $2$-ideal $\cI$ of $\cC$ there is a $2$-category $\cA$ and a $2$-functor
$\Phi:\cC\to \cA$ such that $\mathrm{Ker}(\Phi)=\cI$.
\end{enumerate}
\end{lemma}

\begin{proof}
Claim \eqref{lem11.1} is checked by a direct computation. To prove claim \eqref{lem11.2} just take $\cA=\cC/\cI$
and let $\Phi$ be the natural projection.
\end{proof}

\begin{lemma}\label{lem75}
Let $\cC$ be a fiat category and $\cI$ a $2$-ideal. 
\begin{enumerate}[(i)]
\item\label{lem75.1} The image in $\cC/\cI$ of an indecomposable $1$-morphism in $\cC$
is either indecomposable or zero. 
\item\label{lem75.2} Every left, right or $2$-sided cell of $\cC$ 
descends either to a left, right or $2$-sided cell of $\cC/\cI$, respectively, or to zero.
\end{enumerate}
\end{lemma}

\begin{proof}
Claim \eqref{lem75.1} follows from definitions and the fact that the quotient of a local 
endomorphism algebra of an indecomposable $1$-morphism is either local or zero.

For two indecomposable $1$-morphisms $\mathrm{F}$ and $\mathrm{G}$ we have $\mathrm{F}\leq_L\mathrm{G}$
provided that there exists $\mathrm{H}$ such that $\mathrm{G}$ is isomorphic to a direct summand
of $\mathrm{H}\circ\mathrm{F}$. Let $\Phi:\cC\to\cC/\cI$ be the quotient functor, then
$\Phi(\mathrm{G})$ is either zero or isomorphic to a direct summand of $\Phi(\mathrm{H})\circ\Phi(\mathrm{F})$.
Hence $\Phi(\mathrm{F})\leq_L\Phi(\mathrm{G})$ if $\Phi(\mathrm{G})$ is nonzero and claim \eqref{lem75.2} 
follows in the case of left cells. The other cases are similar.
\end{proof}

\subsection{$\mathcal{J}$-simple fiat categories}\label{s5.2}

Let $\cC$ be a fiat category and $\mathcal{J}$ a non-zero two-sided cell in $\cC$. We will say that $\cC$ is 
{\em $\mathcal{J}$-simple} provided that for every non-trivial $2$-ideal $\cI$ in $\cC$ there exists
$\mathrm{F}\in\mathcal{J}$ such that $\cI$ contains $\mathrm{id}_{\mathrm{F}}$. This is the closest
analogue to the notion of a simple $2$-category which makes sense in our context. Our main result of this subsection 
is the following:

\begin{theorem}\label{thm12}
Let $\cC$ be a fiat category and $\mathcal{J}$ a non-zero two-sided cell in $\cC$. Then there is a unique  
$2$-ideal $\cI$ in $\cC$ such that $\cC/\cI$ is $\mathcal{J}$-simple.
\end{theorem}

To prove this theorem we will need the following lemmata.

\begin{lemma}\label{lem13}
\begin{enumerate}[(i)]
\item\label{lem13.1} If $\cI$ is a $2$-ideal of $\cC$ containing  $\mathrm{id}_{\mathrm{F}}$ for some 
$\mathrm{F}\in\mathcal{J}$, then $\cI$ contains $\mathrm{id}_{\mathrm{G}}$ for any $\mathrm{G}\geq_J\mathrm{F}$.
\item\label{lem13.2} Assume that $\cI$ is a $2$-ideal of $\cC$ generated by $\mathrm{id}_{\mathrm{F}}$ for some 
indecomposable $\mathrm{F}$. For any $1$-morphisms $\mathrm{G}$ and $\mathrm{H}$ let
$\mathrm{Hom}_{\cJ}(\mathrm{G},\mathrm{H})$ denote the subspace of $\mathrm{Hom}_{\cC}(\mathrm{G},\mathrm{H})$
generated by all morphisms which factor through some $\mathrm{K}$ such that $\mathrm{K}\geq_J\mathrm{F}$.
Then $\cJ$ is a $2$-ideal of $\cC$ and  $\cI=\cJ$.
\end{enumerate}
\end{lemma}

\begin{proof}
As $\cI$ is a $2$-ideal containing  $\mathrm{id}_{\mathrm{F}}$, it also contains $\mathrm{id}_{\mathrm{G}}$
for any $\mathrm{G}=\mathrm{H}\circ \mathrm{F}\circ\mathrm{H}'$ (using the horizontal composition of 
$\mathrm{id}_{\mathrm{F}}$ with $\mathrm{id}_{\mathrm{H}}$ on the left and $\mathrm{id}_{\mathrm{H}'}$ on the right).
Composing $\mathrm{id}_{\mathrm{G}}$ with projections onto all direct summand of $\mathrm{G}$ we get that
$\cI$ also contains $\mathrm{id}_{\mathrm{G}'}$ for any direct summand $\mathrm{G}'$ of $\mathrm{G}$. 
Claim \eqref{lem13.1} now follows from the definition of $\geq_J$.

That $\cJ$ is a $2$-ideal of $\cC$ follows directly from the fact that the additive closure of all
$\mathrm{K}$ such that $\mathrm{K}\geq_J\mathrm{F}$ is closed with respect to horizontal composition.
As $\mathrm{id}_{\mathrm{F}}$ is contained in$\cJ$, we have $\cI\subset\cJ$. 
On the other hand, from claim \eqref{lem13.1} we know that $\cI$ contains 
$\mathrm{id}_{\mathrm{K}}$ for any $\mathrm{K}\geq_J\mathrm{F}$. Hence $\cJ\subset\cI$ implying claim \eqref{lem13.2}.
\end{proof}

\begin{remark}\label{rem16}
{\rm 
A $2$-ideal $\cI$ of $\cC$ is called {\em thick} if it is generated by $\mathrm{id}_{\mathrm{F}_i}$ for some collection 
$\{\mathrm{F}_i\}$ of $1$-morphisms in $\cC$ (these kinds of ideals, under the name of ``tensor ideals'' were 
considered, for example, in \cite{Os}). From Lemma~\ref{lem13}\eqref{lem13.2} it follows that thick 
$2$-ideals of $\cC$ are in a natural bijection with anti-chains of the partially ordered set of all two-sided cells of 
$\cC$ with respect to the partial order $\leq_J$.
} 
\end{remark}

\begin{lemma}\label{lem14}
The claim of Theorem~\ref{thm12} is true under the assumption that $\mathcal{J}$ is the unique maximal
non-zero two-sided cell of $\cC$ (with respect to $\leq_J$).
\end{lemma}

\begin{proof}
Let $\cJ$ be a $2$-ideal of $\cC$ that does not contain $\mathrm{id}_{\mathrm{F}}$ for any 
$\mathrm{F}\in\mathcal{J}$. Then for any $\mathrm{F}\in\mathcal{J}$ we have the ideal 
$\mathrm{Hom}_{\cJ}(\mathrm{F},\mathrm{F})$ of the finite-dimensional $\Bbbk$-algebra $\mathrm{Hom}_{\cC(\mathtt{i},\mathtt{j})}(\mathrm{F},\mathrm{F})$ which does not coincide with the latter
algebra. Hence $\mathrm{Hom}_{\cJ}(\mathrm{F},\mathrm{F})$ is contained in the Jacobson radical of
$\mathrm{Hom}_{\cC(\mathtt{i},\mathtt{j})}(\mathrm{F},\mathrm{F})$. Note that the Jacobson radical is stable
under taking sums of subideals.

Let $\cI$ denote the sum of all $2$-ideals of $\cC$ that do not contain any $\mathrm{id}_{\mathrm{F}}$ for 
$\mathrm{F}\in\mathcal{J}$. From the previous paragraph it follows that for $\mathrm{F}\in\mathcal{J}$ the ideal 
$\mathrm{Hom}_{\cI}(\mathrm{F},\mathrm{F})$ of $\mathrm{Hom}_{\cC(\mathtt{i},\mathtt{j})}(\mathrm{F},\mathrm{F})$
is contained in the Jacobson radical of $\mathrm{Hom}_{\cC(\mathtt{i},\mathtt{j})}(\mathrm{F},\mathrm{F})$.
In particular, $\mathrm{Hom}_{\cI}(\mathrm{F},\mathrm{F})$ does not contain $\mathrm{id}_{\mathrm{F}}$. 
We will now show that $\cI$ is the unique $2$-ideal such that $\cC/\cI$ is $\mathcal{J}$-simple.

Indeed, consider the quotient $2$-category $\cC/\cI$. Let $\cI'$ be a $2$-ideal of $\cC/\cI$ which does not contain
any $\mathrm{id}_{\mathrm{F}}$ for  $\mathrm{F}\in\mathcal{J}$. Then the full preimage of $\cI'$ in $\cC$
also has this property and hence is contained in $\cI$ by the construction of $\cI$. Hence $\cI'=\cI$ and
$\cC/\cI$ is $\mathcal{J}$-simple. 

For uniqueness,  let $\cI'$ be another $2$-ideal of $\cC$ such that $\cC/\cI'$ is $\mathcal{J}$-simple. Then 
Lemma~\ref{lem13}\eqref{lem13.1} implies that $\cI'$ does not contain any $\mathrm{id}_{\mathrm{F}}$ for  
$\mathrm{F}\in\mathcal{J}$ and hence $\cI'\subset\cI$. The image of $\cI$ in $\cC/\cI'$ is a $2$-ideal which does not
contain any $\mathrm{id}_{\mathrm{F}}$ for  $\mathrm{F}\in\mathcal{J}$. As  $\cC/\cI'$ is $\mathcal{J}$-simple,
we thus get $\cI'=\cI$. The claim follows.
\end{proof}

\begin{proof}[Proof of Theorem~\ref{thm12}.]
Denote by  $\cJ$ the $2$-ideal of $\cC$ generated by  $\mathrm{id}_{\mathrm{G}}$ for all $1$-morphisms
$\mathrm{G}$ in $\cC$ such that $\mathrm{F}\not\geq_J\mathrm{G}$ for any $\mathrm{F}\in \mathcal{J}$.
If $\cI$ is a $2$-ideal of $\cC$ which does not contain any $\mathrm{id}_{\mathrm{F}}$
for $\mathrm{F}\in\mathcal{J}$, then the image of $\cJ$ in $\cC/\cI$ does not contain any $\mathrm{id}_{\mathrm{F}}$
for $\mathrm{F}\in\mathcal{J}$ either. Hence, if $\cC/\cI$ is $\mathcal{J}$-simple, 
then $\cJ\subset\cI$. This means that without loss of generality we may assume (by going to $\cC/\cJ$) that $\cJ$ 
is trivial. However, in the latter case the definition of $\cJ$ implies that $\mathcal{J}$ is the unique maximal 
two-sided cell of $\cC$. Now the claim follows from Lemma~\ref{lem14}.
\end{proof}

\subsection{Annihilators of cell $2$-representations}\label{s5.3}

In this section we describe annihilators of cell $2$-representations.
The following theorem generalizes \cite[Lemma~3.25]{Ag}.

\begin{theorem}\label{thm15}
Let $\mathcal{J}$ be a two-sided cell in $\cC$ and $\mathcal{L}$ be a left cell in $\mathcal{J}$.
Then the ``image'' $2$-category $\cC/\mathrm{Ker}(\mathbf{C}_{\mathcal{L}})$ of the cell 
$2$-representation $\mathbf{C}_{\mathcal{L}}$ is $\mathcal{J}$-simple.
\end{theorem}

\begin{proof}
Let $\cI$ be the $2$-ideal of $\cC$ given by Theorem~\ref{thm12}. Consider the quotient category $\cA:=\cC/\cI$.
Then $\mathcal{J}$ descends, by Lemma~\ref{lem75}\eqref{lem75.2}, to a two-sided cell for  $\cA$ which we will denote by 
$\mathcal{J}_{\ccA}$. Similarly, we have the left cell $\mathcal{L}_{\ccA}$ of $\cA$ coming from $\mathcal{L}$.
The cell $2$-representation $\mathbf{C}_{\mathcal{L}_{\ccA}}$ of $\cA$ becomes a  $2$-representation of $\cC$
by composing the $2$-functor $\mathbf{C}_{\mathcal{L}_{\ccA}}$ with the quotient $2$-functor $\cC\to\cA$, so by definition $\cI \subset \mathrm{Ker}(\mathbf{C}_{\mathcal{L}_{\cccA}})$. On the other hand $\cA$ is $\mathcal{J}$-simple, so no non-trivial ideal of $\cA$ can annihilate $\mathbf{C}_{\mathcal{L}_{\cccA}}$ showing $\mathrm{Ker}(\mathbf{C}_{\mathcal{L}_{\cccA}})\subset \cI$. This yields $\mathrm{Ker}(\mathbf{C}_{\mathcal{L}_{\cccA}})= \cI$.

Let $\mathtt{i}\in\cC$ be such that $\mathrm{G}_{\mathcal{L}}\in \cC(\mathtt{i},\mathtt{i})$.
Similarly to the proof of Proposition~\ref{prop10}, the unique homomorphism from $\mathbb{P}_{\mathtt{i}}$
to $\mathbf{C}_{\mathcal{L}_{\ccA}}$ sending $\mathbbm{1}_{\mathtt{i}}$ to $L_{\mathrm{G}_{\mathcal{L}_{\cccA}}}$
gives rise to a homomorphism $\Phi:\mathbf{C}_{\mathcal{L}}\to(\overline{\mathbf{C}}_{\mathcal{L}_{\cccA}})_{pr}$.
The abelianization $\overline{\Phi}$ sends the simple top $L$  of 
$0\to \mathrm{G}_{\mathcal{L}}\,L_{\mathrm{G}_{\mathcal{L}}}$ to the corresponding simple object $L'$ in
$\overline{(\overline{\mathbf{C}}_{\mathcal{L}_{\cccA}})_{pr}}(\mathtt{i})$.

From \cite[Proposition~17(b)]{MM} it follows that indecomposable projectives for 
$\overline{\mathbf{C}}_{\mathcal{L}}$ and $\overline{\mathbf{C}}_{\mathcal{L}_{\cccA}}$
have the form $\mathrm{F}\, L$ and $\mathrm{F}\, L'$
for $\mathrm{F}\in \mathcal{L}$ or $\mathrm{F}\in \mathcal{L}_{\ccA}$, respectively. This means that 
$\Phi$ sends indecomposable objects to indecomposable objects and, moreover, \cite[Lemma~21]{MM} also implies
that $\Phi$ is full. 

Furthermore, using adjunction, for $\mathrm{F},\mathrm{G}\in\mathcal{J}$ we have
\begin{displaymath}
\mathrm{Hom}(\mathrm{F}\,L,\mathrm{G}\,L)\cong \mathrm{Hom}(L,\mathrm{F}^*\circ\mathrm{G}\,L).
\end{displaymath}
By \cite[Lemma~19]{MM}, the dimension of the latter space equals the multiplicity of $\mathrm{G}_{\mathcal{L}}$
in $\mathrm{F}^*\circ\mathrm{G}$. Similarly for $\cA$. As the decomposition multiplicities for $\cC$ and $\cA$ 
coincide, it follows that $\Phi$ is an equivalence when restricted to any $\mathtt{j}$. This implies that 
$\mathrm{Ker}(\mathbf{C}_{\mathcal{L}})=\mathrm{Ker}(\mathbf{C}_{\mathcal{L}_{\cccA}})$ which yields
$\mathrm{Ker}(\mathbf{C}_{\mathcal{L}})=\cI$. The claim follows.
\end{proof}

We end this subsection with an example of a cell $2$-representation, which is not strongly simple, but whose 
``image'' $2$-category is simple in the above sense.

\begin{example}
{\rm
Consider \cite[Example~7.1]{MM} in more detail for Dynkin type $B_2$. Keeping the notation from this example, we 
let $\cS$ denote the (strict) $2$-category defined as follows: it has one object $\mathtt{i}$, which we identify 
with (a small category equivalent to) the regular block $\mathcal{O}_0$ of the BGG category $\mathcal{O}$ in type 
$B_2$; its $1$-morphisms are projective functors $\theta$ on $\mathcal{O}_0$; and its $2$-morphisms are natural 
transformations of functors. Indecomposable projective functors are indexed by elements of the Weyl group $W$ of 
type $B_2$, which has elements $\{e,s,t,st,ts,sts,tst,stst\}$. Cells are again given by Kahzdan-Lusztig combinatorics,
the two-sided cells are $\mathcal{J}_e=\{e\}, \mathcal{J}_{s,t}=\{s,t,st,ts,sts,tst\}$ and $\mathcal{J}_{stst}=\{stst\}$.
The middle cell splits into two left cells $\mathcal{L}_1 = \{s,st,sts\}$ and $\mathcal{L}_2= \{t,ts,tst\}$
(note that our left cells are Kazhdan-Lusztig's right cells and vice versa). Right cells
are given by $\mathcal{R}_1=\{s,ts,sts \}$ and $\mathcal{R}_2=\{t,st,tst \}$, so this is not strongly regular. Indeed in
the cell $2$-representation $\overline{\mathbf{C}}_{\mathcal{L}_1}$, which is generated by $L_s$ for the Duflo involution
$s$, the simple module $L_{st}$ is annihilated by $\theta_s$ (and hence by $\theta_{sts}$ as well), and is mapped to 
the projective object indexed by $st$ in $\overline{\mathbf{C}}_{\mathcal{L}_1}$ by $\theta_t$. Applying $\theta_s$ 
to this projective object, we obtain the direct sum of the two projectives in $\overline{\mathbf{C}}_{\mathcal{L}_1}$
labeled by $s$ and $sts$. One can check by explicit computations that it is impossible to split this direct sum using 
only morphisms in the image of the action of $\cS$. As the functor $\theta_s\theta_t$ is indecomposable, this implies
that $\overline{\mathbf{C}}_{\mathcal{L}_1}$ is not strongly simple. On the other hand, Theorem~\ref{thm15} implies 
that $\cS/\mathrm{Ker}(\mathbf{C}_{\mathcal{L}_1})$ is $ \mathcal{J}_{s,t}$-simple. 
}
\end{example}

\subsection{Left $2$-ideals}\label{s5.5}

Let $\cA$ be a $\Bbbk$-linear $2$-category. A {\em left $2$-ideal} $\cI$ of $\cA$ consists of the same objects as $\cA$, 
and for each $\mathtt{i},\mathtt{j}\in \cI$ an ideal $\cI(\mathtt{i},\mathtt{j})$ in $\cA(\mathtt{i},\mathtt{j})$ such 
that $\cI$ is stable under left horizontal multiplication with both $1$- and $2$-morphisms in $\cA$. For example,
we can view $\cA(\mathtt{k},{}_-)$ as a left $2$-ideal $\cI_{\mathtt{k}}:=\cI^{\cA}_{\mathtt{k}}$ in $\cA$ in the following way:
\begin{displaymath}
\cI_{\mathtt{k}}(\mathtt{i},\mathtt{j}):=
\begin{cases}
\cA(\mathtt{k},\mathtt{j}),& \mathtt{i}=\mathtt{k};\\
0,&\text{else}.
\end{cases}
\end{displaymath}
It is easy to see that $\cI_{\mathtt{k}}$ is generated, as a left $2$-ideal, by the $2$-morphism 
$\mathrm{id}_{\mathbbm{1}_{\mathtt{k}}}$.

Let $\cC$ be a fiat category. Given a left $2$-ideal $\cJ$ contained in $\cI_{\mathtt{i}}$, we can define 
the quotient $2$-representation  $\mathbb{P}_{\mathtt{i}}/\cJ$ by letting 
\begin{displaymath}
(\mathbb{P}_{\mathtt{i}}/\cJ)(\mathtt{j}):=\mathbb{P}_{\mathtt{i}}(\mathtt{j})/\cJ(\mathtt{i},\mathtt{j})
\end{displaymath}
with the action of $\cC$ induced from that on $\mathbb{P}_{\mathtt{i}}$. It is easy to check that 
$\mathbb{P}_{\mathtt{i}}/\cJ$ is finitary.

Left $2$-ideals appear naturally as annihilators of objects for $2$-representations. Let
$\mathbf{M}$ be a $2$-representation of $\cC$ and $X\in\mathbf{M}(\mathtt{i})$. Define
$\mathrm{Ann}_{\ccC}(X)$ as the set of all $2$-morphisms $\alpha$ such that either $\alpha=0$ or
$\mathbf{M}(\alpha)_X$ is defined and $\mathbf{M}(\alpha)_X=0$. Then it is easy to see that 
$\mathrm{Ann}_{\ccC}(X)$ has the natural structure of a left $2$-ideal of $\cC$. Note that
$\mathrm{Ann}_{\ccC}(X)\subset \cI_{\mathtt{i}}$ by definition.

\begin{proposition}\label{prop42}
Assume that $\cC$ has a unique maximal two-sided cell $\mathcal{J}$ and let $\mathcal{L}$ be 
a left cell in $\mathcal{J}$. Then for $\mathtt{i}:=\mathtt{i}_{\mathcal{L}}$ there is a unique
maximal left $2$-ideal $\cI$ of $\cC$ contained in $\cI_{\mathtt{i}}$ such that 
$\cI$ does not contain $\mathrm{id}_{\mathrm{G}_{\mathcal{L}}}$.
\end{proposition}

\begin{proof}
Let $\cI\subset \cI_{\mathtt{i}}$ be a left $2$-ideal
of $\cC$ which does not contain $\mathrm{id}_{\mathrm{G}_{\mathcal{L}}}$. 
Similarly to the proof of Lemma~\ref{lem13} one shows that $\cI$ does not contain any $\mathrm{id}_{\mathrm{F}}$
for $\mathrm{F}\in \mathcal{L}$. This implies that the part of $\cI$ inside the local algebra
$\mathrm{End}_{\ccC(\mathtt{i},\mathtt{j})}(\mathrm{F})$ is properly contained in the radical.
This property is preserved if one consider the sum of all $\cI$'s. The claim follows.
\end{proof}

\subsection{An alternative construction of cell modules}\label{s5.4}

Let $\mathcal{J}$ be a two-sided cell in $\cC$ and $\mathcal{L}$ a left cell in $\mathcal{J}$. Let 
$\cI$ be the $2$-ideal in $\cC$ given by Theorem~\ref{thm12} (i.e. $\cA:=\cC/\cI$ is $\mathcal{J}$-simple). 
For $\mathtt{i}:=\mathtt{i}_{\mathcal{L}}$ let $\cJ$ be the unique maximal left $2$-ideal of $\cA$ contained
in $\cI^{\ccA}_{\mathtt{i}}$ which does not contain $\mathrm{id}_{\mathrm{F}}$ for any $\mathrm{F}\in \mathcal{L}$.
Then the $2$-representation $\mathbb{P}_{\mathtt{i}}^{\ccA}/\cJ$ of $\cA$ has the natural structure of a 
$2$-representation of $\cC$ via the quotient $2$-functor $\cC\tto \cA$. 
For $\mathtt{j}\in\cC$ let $\mathbf{D}_{\mathcal{L}}(\mathtt{j})$ denote the full subcategory of 
$(\mathbb{P}_{\mathtt{i}}^{\ccA}/\cJ)(\mathtt{j})$ containing all $1$-morphisms $\mathrm{F}$
from $\cA(\mathtt{i},\mathtt{j})$ corresponding to $\mathcal{L}$. Then $\mathbf{D}_{\mathcal{L}}$
inherits the structure of a $2$-representation of $\cC$ by restriction.

\begin{proposition}\label{prop16}
The $2$-rep\-re\-sen\-ta\-ti\-ons $\mathbf{D}_{\mathcal{L}}$ 
and $\mathbf{C}_{\mathcal{L}}$ are elementary equivalent.
\end{proposition}

\begin{proof}
From Theorem~\ref{thm15} it follows that $\cI$ annihilates $L_{\mathrm{G}_{\mathcal{L}}}$ in $\overline{\mathbb{P}}_{\mathtt{i}}$. Therefore, 
$\mathbf{C}_{\mathcal{L}} = (\overline{\mathbb{P}}_{\mathtt{i}})_{\mathrm{G}_{\mathcal{L}}\,L_{\mathrm{G}_{\mathcal{L}}}}$ is indeed a $2$-representation of $\cA$. Mapping $\mathbbm{1}_{\mathtt{i}}$ to $L_{\mathrm{G}_{\mathcal{L}}}$ extends to a homomorphism $\Phi$ from $\mathbb{P}_{\mathtt{i}}^{\ccA}$ to
$(\overline{\mathbb{P}}_{\mathtt{i}})_{L_{\mathrm{G}_{\mathcal{L}}}}$ (the latter considered as a
$2$-representation of $\cA$). Let $L$ be the simple head of the indecomposable projective object 
\begin{displaymath}
0\to \mathrm{G}_{\mathcal{L}}\,L_{\mathrm{G}_{\mathcal{L}}}\in
\overline{(\overline{\mathbb{P}}_{\mathtt{i}})_{L_{\mathrm{G}_{\mathcal{L}}}}}(\mathtt{i}).
\end{displaymath}

We obviously have  $\cJ':=\mathrm{Ann}_{\ccA}(L)\subset\cJ$. We further claim that for
any $\mathrm{F},\mathrm{G}\in\mathcal{L}\cap\cC(\mathtt{i},\mathtt{j})$ we even have
\begin{equation}\label{eq17}
\mathrm{Hom}_{\ccJ'(\mathtt{i},\mathtt{j})}(\mathrm{F},\mathrm{G})=\mathrm{Hom}_{\ccJ(\mathtt{i},\mathtt{j})}(\mathrm{F},\mathrm{G}).
\end{equation}

Indeed, by \cite[Lemma~21]{MM}, $\mathrm{Hom}_{\ccA(\mathtt{i},\mathtt{j})}(\mathrm{F},\mathrm{G})/
\mathrm{Hom}_{\ccJ'}(\mathrm{F},\mathrm{G})$ is mapped bijectively onto the space
$\mathrm{Hom}_{\mathbf{C}_{\mathcal{L}}(\mathtt{j})}
(\mathrm{F}\,L_{\mathrm{G}_{\mathcal{L}}},\mathrm{G}\,L_{\mathrm{G}_{\mathcal{L}}})$, so, 
by \cite[Lemma~19]{MM}, the codimension of the left space
in $\mathrm{Hom}_{\ccA(\mathtt{i},\mathtt{j})}(\mathrm{F},\mathrm{G})$ equals the multiplicity
$m$ of $\mathrm{G}_{\mathcal{L}}$ in $\mathrm{G}^*\circ\mathrm{F}$.

On the other hand, using adjunction, we obtain 
\begin{displaymath}
\begin{array}{rcl}
\mathrm{Hom}_{\mathbf{D}_{\mathcal{L}}(\mathtt{j})}(\mathrm{F},\mathrm{G})&=&
\mathrm{Hom}_{(\mathbb{P}_{\mathtt{i}}^{\cccA}/\ccJ)(\mathtt{j})}(\mathrm{F},\mathrm{G})\\&\cong&
\mathrm{Hom}_{(\mathbb{P}_{\mathtt{i}}^{\cccA}/\ccJ)(\mathtt{i})}(\mathrm{G}^*\circ\mathrm{F},
\mathbbm{1}_{\mathtt{i}}).
\end{array}
\end{displaymath}
By \cite[Proposition~17(a)]{MM}, there is $0\neq\alpha\in 
\mathrm{Hom}_{\mathbb{P}_{\mathtt{i}}^{\cccA}(\mathtt{i})}(\mathrm{G}_{\mathcal{L}},
\mathbbm{1}_{\mathtt{i}})$ such that the $2$-morphism $\mathrm{F}(\alpha):\mathrm{F}\circ 
\mathrm{G}_{\mathcal{L}}\to \mathrm{F}$ is an epimorphism for any $\mathrm{F}\in\mathcal{L}$, 
so $F$ appears as a direct summand of $\mathrm{F}\circ \mathrm{G}_{\mathcal{L}}$. If $\cJ$ contained $\alpha$, 
being a left $2$-ideal, $\cJ$ would contain $\mathrm{F}(\alpha)$ as well, and hence, composing it 
with an inclusion from a direct summand, $\cJ$ would contain $\mathrm{id}_{\mathrm{F}}$, a contradiction. 
Hence $\alpha\not\in\cJ$, which yields
$\mathrm{Hom}_{(\mathbb{P}_{\mathtt{i}}^{\cccA}/\ccJ)(\mathtt{i})}(\mathrm{G}_{\mathcal{L}},
\mathbbm{1}_{\mathtt{i}})\neq 0$, so $\dim \mathrm{Hom}_{(\mathbb{P}_{\mathtt{i}}^{\cccA}/\ccJ)(\mathtt{i})}(\mathrm{G}^*\circ\mathrm{F},
\mathbbm{1}_{\mathtt{i}}) \geq m$, and thus 
\begin{displaymath}
\dim \mathrm{Hom}_{\mathbf{D}_{\mathcal{L}}(\mathtt{j})}(\mathrm{F},\mathrm{G})\geq m=
\dim \mathrm{Hom}_{\mathbf{C}_{\mathcal{L}}(\mathtt{j})}(\mathrm{F}\,L_{\mathrm{G}_{\mathcal{L}}},
\mathrm{G}\,L_{\mathrm{G}_{\mathcal{L}}}).
\end{displaymath}
Taking codimensions in $\mathrm{Hom}_{\ccA(\mathtt{i},\mathtt{j})}(\mathrm{F},\mathrm{G})$, this implies
$$\dim\mathrm{Hom}_{\ccJ(\mathtt{i},\mathtt{j})}(\mathrm{F},\mathrm{G}) \leq \dim \mathrm{Hom}_{\ccJ'(\mathtt{i},\mathtt{j})}(\mathrm{F},\mathrm{G})$$
and hence proves \eqref{eq17}. It follows that  $\Phi$  factors over $\mathbb{P}_{\mathtt{i}}^{\ccA}/\cJ$ and
restricts to an elementary equivalence from $\mathbf{D}_{\mathcal{L}}$ to $\mathbf{C}_{\mathcal{L}}$.
\end{proof}

\section{Further examples and constructions}\label{s6}

\subsection{$\mathfrak{sl}_2$-categorification}\label{s6.1}

Fix $n\in\{0,1,2,\dots\}$. Denote by $\mathtt{C}_n$ the complex coinvariant algebra of the symmetric group $S_n$, 
that is the quotient $\mathbb{C}[x_1,x_2,\dots,x_n]/I$, where the ideal $I$ is generated by all homogeneous 
$S_n$-symmetric polynomials of positive degree. For $i=1,2,\dots,n-1$ let $s_i$ denote the transposition $(i,i+1)$ 
and $S_n^i$ the subgroup of $S_n$ generated by all $s_j$, $j\neq i$. Set $\mathtt{C}_n^0=\mathtt{C}_n^n:=\mathbb{C}$ 
and for $i=1,2,\dots,n-1$ define $\mathtt{C}_n^i$ and the subalgebra of $S_n^i$-invariants in $\mathtt{C}_n$ (note 
that $\dim \mathtt{C}_n^i=\binom{n}{i}$). Similarly, for different $i,j$ we denote by $S_n^{i,j}$ the subgroup of 
$S_n$ generated by all $s_k$, $k\neq i,j$, and define $\mathtt{C}_n^{i,j}$ to be the subalgebra $S_n^{i,j}$-invariants 
in $\mathtt{C}_n$. It is well known that both $\mathtt{C}_n$, $\mathtt{C}_n^i$ and $\mathtt{C}_n^{i,j}$ are symmetric 
algebras and that $\mathtt{C}_n^{i,i+1}$ is free both as a $\mathtt{C}_n^{i}$- and $\mathtt{C}_n^{i+1}$-module (this
follows from the geometric interpretation of the coinvariant algebra as the cohomology algebra of a flag variety, 
see e.g \cite{Hi}).

Define the $2$-category $\cB_n$ via its defining representation as follows: The objects of $\cB_n$ are
$\mathtt{0},\mathtt{1},\dots,\mathtt{n}$ where we identify the object $\mathtt{i}$ with some small category
$\mathcal{C}_i$ equivalent to the category of $\mathtt{C}_n^i$-modules. As generating $1$-morphisms between 
$\mathcal{C}_i$ and $\mathcal{C}_{i+1}$ we take the functors
\begin{equation}\label{eq3}
\mathrm{F}_{i}^{(1)}:=\mathrm{Res}_{\mathtt{C}_n^{i+1}}^{\mathtt{C}_n^{i,i+1}}\circ
\mathrm{Ind}_{\mathtt{C}_n^i}^{\mathtt{C}_n^{i,i+1}} :\mathcal{C}_i\to  
\mathcal{C}_{i+1}\,\text{ and }\,
\mathrm{E}_{i+1}^{(1)}:=\mathrm{Res}_{\mathtt{C}_n^{i}}^{\mathtt{C}_n^{i,i+1}}\circ
\mathrm{Ind}_{\mathtt{C}_n^{i+1}}^{\mathtt{C}_n^{i,i+1}} :\mathcal{C}_{i+1}\to  
\mathcal{C}_{i}.
\end{equation}
We let $\cB_n$ be the minimal full fully additive $2$-subcategory of the endomorphism $2$-category of the
$\mathcal{C}_i$'s containing these generating functors and closed with respect to isomorphism of $1$-morphisms.
Thus, $2$-morphisms in $\cB_n$ are all natural transformations of functors.

\begin{remark}\label{rem17}
{\rm 
The category $\cB_n$ has an alternative description via the BGG category $\mathcal{O}$ for $\mathfrak{gl}_n$.
Let $\mathfrak{gl}_n=\mathfrak{n}_-\oplus\mathfrak{h}\oplus\mathfrak{n}_+$ be the standard triangular decomposition of
$\mathfrak{gl}_n$. Fix a basis $\{e_{ii}\}\in \mathfrak{h}$ and the corresponding dual basis in $\mathfrak{h}^*$.
Denote by $\mathcal{O}^{i}$ the block of the category $\mathcal{O}$ associated to this triangular decomposition
(see \cite{BGG,Hu}) containing simple highest weight modules $L(\lambda)$ (with highest weight $\lambda-\rho$ where
$\rho$ is the half sum of all positive roots), where $\lambda\in \mathfrak{h}^*$ is a $0$-$1$ vector (with respect to 
the basis chosen above) with exactly $i$ components equal to $0$. The block $\mathcal{O}^{i}$ has a unique 
(up to isomorphism)  indecomposable projective-injective module, which we denote by  $P_i$, and 
$\mathrm{End}_{\mathcal{O}}(P_i)\cong \mathtt{C}_n^{i}$. The full subcategory $\mathcal{P}_i$ of $\mathcal{O}^{i}$
consisting of all modules with projective-injective presentation is then equivalent to 
$\mathtt{C}_n^{i}\text{-}\mathrm{mod}$. This establishes a connection to $\cB_n$ (one has to choose appropriate small
versions of all $\mathcal{P}_i$'s). Generating $1$-morphisms from \eqref{eq3} are realized via projective functors
between different blocks of $\mathcal{O}$ in the sense of \cite{BG} (in this case these are just tensoring with the 
natural representation of $\mathfrak{gl}_n$ or its dual, followed by the projection onto an appropriate block). 
We refer to \cite{BFK,FKS} for further details.
}
\end{remark}

\begin{remark}\label{rem18}
{\rm 
Some other versions of $\cB_n$ can be found implicitly in \cite{CR,La} (and some other papers). One essential 
difference with all these papers is that they also prescribe generating $2$-morphisms while we just take all of them
(following \cite{BFK}).
}
\end{remark}

Let $i\in\{0,1,\dots,n\}$ and $k,l\in\{0,1,\dots,n\}$ be such that $i+k,i-l\in \{0,1,\dots,n\}$. Similarly to the above
define
\begin{displaymath}
\mathrm{F}_{i}^{(k)}:=\mathrm{Res}_{\mathtt{C}_n^{i+k}}^{\mathtt{C}_n^{i,i+k}}\circ
\mathrm{Ind}_{\mathtt{C}_n^i}^{\mathtt{C}_n^{i,i+k}} :\mathcal{C}_i\to  
\mathcal{C}_{i+k}\,\text{ and }\,
\mathrm{E}_{i+k}^{(k)}:=\mathrm{Res}_{\mathtt{C}_n^{i}}^{\mathtt{C}_n^{i,i+k}}\circ
\mathrm{Ind}_{\mathtt{C}_n^{i+k}}^{\mathtt{C}_n^{i,i+k}} :\mathcal{C}_{i+k}\to  
\mathcal{C}_{i}.
\end{displaymath}
It is convenient to write $\mathrm{F}_{i}^{(0)}$ and $\mathrm{E}_{i}^{(0)}$ for $\mathbbm{1}_{\mathtt{i}}$.
From the facts that $\mathtt{C}_n$, $\mathtt{C}_n^i$ and $\mathtt{C}_n^{i,j}$ are symmetric and that 
$\mathtt{C}_n^{i,i+1}$ is free both as a $\mathtt{C}_n^{i}$- and $\mathtt{C}_n^{i+1}$-module it follows that 
the functor $\mathrm{F}_{i}^{(k)}$ is both left and right adjoint (i.e. biadjoint) to $\mathrm{E}_{i+k}^{(k)}$.
Indecomposable $1$-morphisms in $\cB_n$ are given by the following claim proved in \cite[3.1.3]{BFK}:

\begin{proposition}\label{prop17}
Let $i,j\in\{0,1,\dots,n\}$. The following is a complete and irredundant list of indecomposable $1$-morphisms 
in $\cB_n(\mathtt{i},\mathtt{j})$, up to isomorphism:
\begin{enumerate}[(a)]
\item\label{prop17.1} $\mathrm{F}_i^{(j-i)}, \mathrm{F}_{i+1}^{(j-i+1)}\mathrm{E}_i^{(1)},
\mathrm{F}_{i+2}^{(j-i+2)}\mathrm{E}_i^{(2)},\dots,\mathrm{F}_{0}^{(j)}\mathrm{E}_i^{(i)}$, if $i<j$ and $i< n-j$. 
\item\label{prop17.2} $\mathrm{F}_i^{(j-i)}, \mathrm{E}_{j+1}^{(1)}\mathrm{F}_{i}^{(j-i+1)},
\mathrm{E}_{j+2}^{(2)}\mathrm{F}_{i}^{(j-i+2)},\dots,\mathrm{E}_n^{(n-j)}\mathrm{F}_{i}^{(n-i)}$,  
if $i<j$ and $i \geq n-j$. 
\item\label{prop17.3} $\mathrm{E}_i^{(i-j)}, \mathrm{E}_{i+1}^{(i-j+1)}\mathrm{F}_{i}^{(1)},
\mathrm{E}_{i+2}^{(i-j+2)}\mathrm{F}_{i}^{(2)},\dots,\mathrm{E}_{n}^{(n-j)}\mathrm{F}_{i}^{(n-i)}$ ,
if $i>j$ and $j\geq n-i$. 
\item\label{prop17.4} $\mathrm{E}_i^{(i-j)}, \mathrm{F}_{j-1}^{(1)}\mathrm{E}_{i}^{(i-j+1)},
\mathrm{F}_{j-2}^{(2)}\mathrm{E}_{i}^{(i-j+2)},\dots,\mathrm{F}_{0}^{(j)}\mathrm{E}_{i}^{(i)}$, 
if $i>j$ and $i< n-j$. 
\item\label{prop17.5} $\mathbbm{1}_{\mathtt{i}}, \mathrm{F}_{i-1}^{(1)}\mathrm{E}_{i}^{(1)},
\mathrm{F}_{i-2}^{(2)}\mathrm{E}_{i}^{(2)},\dots,\mathrm{F}_{0}^{(i)}\mathrm{E}_{i}^{(i)}$, if $i=j$ and $i< n-i$.
\item\label{prop17.6} $\mathbbm{1}_{\mathtt{i}}, \mathrm{E}_{i+1}^{(1)}\mathrm{F}_{i}^{(1)},
\mathrm{E}_{i+2}^{(2)}\mathrm{F}_{i}^{(2)},\dots,\mathrm{E}_{n}^{(n-i)}\mathrm{F}_{n-i}^{(n-i)}$, 
if $i=j$ and $i\geq n-i$.
\end{enumerate}
\end{proposition}

\begin{corollary}\label{cor18}
The $2$-category $\cB_n$ is fiat.
\end{corollary}

\begin{proof}
From Proposition~\ref{prop17} we see that $\cB_n$ has only finitely many indecomposable $1$-morphisms up to isomorphism.
The weak involution is given by taking the adjoint functor, which also implies existence of adjunction morphisms.
The rest is clear.
\end{proof}

Set $\mathrm{E}_0^{(1)}=\mathrm{E}_{n+1}^{(1)}=\mathrm{F}_{-1}^{(1)}=\mathrm{F}_{n}^{(1)}=0$. Then \cite[Theorem~1]{BFK}
asserts that for every $i\in\{0,1,\dots,n\}$ there is an isomorphism of functors as follows:
\begin{equation}\label{eq55}
(\mathrm{E}_{i+1}^{(1)}\circ \mathrm{F}_{i}^{(1)})\oplus\mathbbm{1}_{\mathtt{i}}^{\oplus(i)}\cong 
(\mathrm{F}_{i-1}^{(1)}\circ \mathrm{E}_{i+1}^{(1)})\oplus \mathbbm{1}_{\mathtt{i}}^{\oplus(n-i)}.
\end{equation}
Furthermore, for every $i\in\{0,1,\dots,n\}$, $k=1,2,\dots,i$ and $l=1,2,\dots,n-i$ there are isomorphisms of 
functors as follows:
\begin{equation}\label{eq65}
\mathrm{E}_{i+(k-1)}^{(1)}\circ\cdots\circ\mathrm{E}_{i+1}^{(1)}\circ\mathrm{E}_{i}^{(1)}\cong
\big(\mathrm{E}_{i}^{(k)}\big)^{\oplus k!},\quad
\mathrm{F}_{i-(k-1)}^{(1)}\circ\cdots\circ\mathrm{F}_{i-1}^{(1)}\circ\mathrm{F}_{i}^{(1)}\cong
\big(\mathrm{F}_{i}^{(k)}\big)^{\oplus k!}.
\end{equation}
This can be used to describe cells in $\cB_n$ as follows: 
\begin{itemize}
\item for $i\in\{0,1,\dots, \lfloor\frac{n}{2}\rfloor\}$ let $\mathcal{J}_i$ denote the two-sided cell of $\cB_n$ 
containing $\mathbbm{1}_{\mathtt{i}}$;
\item for $(i,k)$ such that $i\in\{0,1,\dots \lfloor\frac{n}{2}\rfloor\}$ and $k\in\{0,1,\dots,i\}$ let
$\mathcal{L}_{(i,k)}$ denote the left cell of $\cB_n$  containing $\mathrm{F}_{i-k}^{(k)}\mathrm{E}_{i}^{(k)}$; 
\item for $(i,k)$ such that $i\in\{\lfloor\frac{n}{2}\rfloor+1,\dots,n\}$ and $k\in\{0,1,\dots,n-i\}$ let
$\mathcal{L}_{(i,k)}$ denote the left cell of $\cB_n$  containing $\mathrm{E}_{i+k}^{(k)}\mathrm{F}_{i}^{(k)}$; 
\item for $(i,k)$ such that $i\in\{0,1,\dots \lfloor\frac{n}{2}\rfloor\}$ and $k\in\{0,1,\dots,i\}$ let
$\mathcal{R}_{(i,k)}$ denote the right cell of $\cB_n$  containing $\mathrm{F}_{i-k}^{(k)}\mathrm{E}_{i}^{(k)}$; 
\item for $(i,k)$ such that $i\in\{\lfloor\frac{n}{2}\rfloor+1,\dots,n\}$ and $k\in\{0,1,\dots,n-i\}$ let
$\mathcal{R}_{(i,k)}$ denote the right cell of $\cB_n$  containing $\mathrm{E}_{i+k}^{(k)}\mathrm{F}_{i}^{(k)}$.
\end{itemize}

\begin{proposition}\label{prop19}
\begin{enumerate}[(i)]
\item\label{prop19.1} The two-sided cells $\mathcal{J}_i$ defined above give a complete and irredundant list of
two-sided cells.
\item\label{prop19.2} The left cells $\mathcal{L}_{(i,k)}$ defined above give a complete and irredundant list of
left cells. 
\item\label{prop19.3} The right cells $\mathcal{R}_{(i,k)}$ defined above give a complete and irredundant list of
right cells. 
\item\label{prop19.4} All two-sided cells of  $\cB_n$ are strongly regular and satisfy condition \eqref{eq2}.
\end{enumerate}
\end{proposition}

\begin{proof}
Claims \eqref{prop19.1}--\eqref{prop19.3} follow from Proposition~\ref{prop17} and isomorphisms \eqref{eq55} 
and \eqref{eq65} by direct calculation. Similarly one checks that all two-sided cells of  $\cB_n$ are strongly regular,
so we are only left to check condition \eqref{eq2}. 

Consider $\mathcal{R}_{(i,k)}$ for some fixed pair $(i,k)$. Assume $i\in\{0,1,\dots \lfloor\frac{n}{2}\rfloor\}$ 
and $k\in\{0,1,\dots,i\}$ (the other case is dealt with by a similar argument). Then, using Proposition~\ref{prop17} 
and isomorphisms \eqref{eq55} and \eqref{eq65} one shows that $\mathcal{R}_{(i,k)}$ consists exactly of the
following morphisms (here $s:=\lfloor\frac{n}{2}\rfloor$): 
\begin{gather*}
\mathrm{F}_{k}^{(i-k)},\mathrm{F}_{k}^{(i-k)}\mathrm{E}_{k+1}^{(1)},\dots,
\mathrm{F}_{k}^{(i-k)}\mathrm{E}_{s}^{(s-k)},\\
\mathrm{E}_{n-k}^{(n-k-i)},\mathrm{E}_{n-k}^{(n-k-i)}\mathrm{F}_{n-k-1}^{(1)},\dots,
\mathrm{E}_{n-k}^{(n-k-i)}\mathrm{F}_{s+1}^{(n-k-(s+1))}.
\end{gather*}
Using \eqref{eq55} one checks that the value of $m_{\mathrm{F}}$ on all these $1$-morphisms equals the multiplicity
of $\mathbbm{1}_{\mathtt{k}}$ in $\mathrm{E}_{i}^{(i-k)}\mathrm{F}_{k}^{(i-k)}$.
\end{proof}

\begin{corollary}\label{cor20}
The defining representation of $\cB_n$ is equivalent to $\overline{\mathbf{C}}_{\mathcal{L}_{(0,0)}}$. 
\end{corollary}

\begin{proof}
Denote the  defining representation of $\cB_n$ by $\mathbf{M}$. Sending $\mathbbm{1}_{\mathtt{0}}$ to the projective 
simple module in $\mathcal{C}_0$ restricts to a homomorphism from ${\mathbf{C}}_{\mathcal{L}_{(0,0)}}$ to 
$\mathbf{M}_{pr}$ which is easily seen to be an elementary equivalence. The claim now follows from Theorem~\ref{thm9}.
\end{proof}

It is easy to check that the decategorification of $\overline{\mathbf{C}}_{\mathcal{L}_{(0,0)}}$ gives
a simple $n+1$-dimensional $\mathfrak{sl}_2$-module (the decategorification of $\cB_n$ gives the image of 
Lusztig's idempotent completion $\dot{U}_{\mathfrak{sl}_2}$ of $U(\mathfrak{sl}_2)$ acting on this 
simple $n+1$-dimensional $\mathfrak{sl}_2$-module). All this is spelled out e.g. in \cite{BFK,FKS}.

\begin{corollary}\label{cor21}
The category $\cB_n$ is $\mathcal{J}_0$-simple. 
\end{corollary}

\begin{proof}
This follows from Corollary~\ref{cor20} and Theorem~\ref{thm15} as the defining representation is obviously faithful. 
\end{proof}

The latter implies the following recursion for $\cB_n$ (for appropriate choices of the $\mathcal{C}_i$'s in 
the corresponding defining representations):

\begin{theorem}\label{thm22}
The $2$-categories $\cB_{n-2}$ and $\cB_n/\mathrm{Ker}(\mathbf{C}_{\mathcal{L}_{1}})$ are isomorphic.
\end{theorem}

\begin{proof}
We have natural surjective homomorphisms of algebras $\mathtt{C}_{n}\tto\mathtt{C}_{n-2}$ given by
forgetting $s_1$ and $s_{n-1}$, evaluating $x_1$ and $x_n$ at $0$, and sending other $x_i$ to $x_{i-1}$.
The gives rise to surjections for all corresponding subalgebras of invariants.
For appropriate choices of the $\mathcal{C}_i$'s this induces a surjective $2$-functor $\Phi:\cB_n\tto \cB_{n-2}$.
By Corollary~\ref{cor21}, $\cB_{n-2}$ is $\mathcal{J}_0$-simple (for its index $0$ that corresponds to index $1$
for the category $\cB_n$), which implies that the $2$-ideal $\mathrm{Ker}(\mathbf{C}_{\mathcal{L}_{1}})$
of $\cB_n$ coincides with $\mathrm{Ker}(\Phi)$. The claim follows.
\end{proof}

\begin{remark}\label{rem23}
{\em 
The ideal $\mathrm{Ker}(\mathbf{C}_{\mathcal{L}_{1}})$ obviously contains the $2$-morphism $\mathrm{id}_{\mathbbm{1}_{\mathtt{0}}}$ and hence also $\mathrm{id}_{\mathrm{F}}$ for any $\mathrm{F}$ in the
$2$-sided cell of $\mathbbm{1}_{\mathtt{0}}$. However, it is easy to check that 
$\mathrm{Ker}(\mathbf{C}_{\mathcal{L}_{1}})$ is not generated by $\mathrm{id}_{\mathbbm{1}_{\mathtt{0}}}$ in general.
}
\end{remark}

\subsection{$2$-Schur algebra}\label{s6.3}

Let $n,r\in \mathbb{N}$. Consider the Lie algebra $\mathfrak{gl}_{r}$ with a fixed standard triangular decomposition
$\mathfrak{n}_{-}\oplus \mathfrak{h}\oplus \mathfrak{n}_{+}$. For $\lambda\in\mathfrak{h}^*$ let $M(\lambda)$ denote
the Verma modules with highest weight $\lambda-\rho$, where $\rho$ is the half of the sum of all positive roots.
For a dominant $\lambda\in\mathfrak{h}^*$ let $\mathcal{O}_{\lambda}$ be the block of the BGG category $\mathcal{O}$
(\cite{BGG}) associated with this triangular decomposition, containing $M(\lambda)$, see \cite{Ma} for details.

We fix in $\mathfrak{h}^*$ the basis, dual to the basis of $\mathfrak{h}$ consisting of matrix units. Using this
basis, we identify $\mathfrak{h}^*$ with $\mathbb{C}^r$. Denote by $\mathtt{N}$ the set of all vectors
$\mathbf{v}:=(v_1,v_2,\dots,v_r)\in \mathbb{C}^r$ such that $v_i\in\{1,2,\dots,n\}$ for all $i$. Let 
$\mathtt{N}_d$ denote the subset of $\mathtt{N}$ consisting of all {\em dominant} vectors, that is all 
$(v_1,v_2,\dots,v_r)$ such that $v_1\geq v_2\geq\dots\geq v_r$.

Consider the $2$-category $\cD_{n,r}$ defined via its defining representation as follows: objects of 
$\cD_{n,r}$ are $\mathbf{v}\in \mathtt{N}_d$, which we identify with some small category 
$\mathcal{Q}_{\mathbf{v}}$, equivalent to ${\mathcal{O}}_{\mathbf{v}}$; $1$-morphisms are projective
functors in the sense of \cite{BG} and $2$-morphisms are natural transformations of functors. This is a fiat
category, see \cite[7.2]{MM}. By \cite[Section~9]{MS2}, the decategorification of the defining representation
of $\cD_{n,r}$ gives the defining representation of the Schur algebra $S(n,r)$ (see also \cite{MS3}). 
This allows us to consider $\cD_{n,r}$ as a categorification (a $2$-analogue) of $S(n,r)$ and here we 
describe the combinatorics of $\cD_{n,r}$ in more details. We refer the reader to \cite{Mathas} for more details
on the Schur algebra associated to the symmetric group.

For $\mathbf{v}\in \mathtt{N}_d$ let $\mathtt{N}_{\mathbf{v}}$ and $S_{\mathbf{v}}$ denote the orbit of 
$\mathbf{v}$ under the action of the symmetric group $S_r$ and the stabilizer of $\mathbf{v}$ with respect to this
action, respectively. By \cite{BG}, for $\mathbf{v},\mathbf{u}\in \mathtt{N}_d$ indecomposable projective functors 
from $\mathcal{Q}_{\mathbf{v}}$ to $\mathcal{Q}_{\mathbf{u}}$ correspond  bijectively to 
{\em $\mathbf{v}$-antidominant} elements in $\mathtt{N}_{\mathbf{u}}$ (that is $\mathbf{x}\in \mathtt{N}_{\mathbf{u}}$ 
such that for any $i\in\{1,2,\dots,r-1\}$ the equality $v_i=v_{i+1}$ implies the inequality $x_i\leq x_{i+1}$),
alternatively, to double cosets $S_{\mathbf{v}}\hspace{-1mm}\setminus S_r/S_{\mathbf{u}}$.  Moreover, by \cite{BG} 
the decategorifications of indecomposable projective functors are linearly independent. As the indexing set 
for indecomposable projective functors
coincides with the indexing set of the standard basis for $S(n,r)$, we conclude that the decategorification 
of $\cD_{n,r}$ is isomorphic to $S(n,r)$. The distinguished basis of $S(n,r)$ corresponding to indecomposable
projective functors, given by this construction, is Du's canonical basis of $S(n,r)$ defined in \cite{Du}.

The Robinson-Schensted-Knuth algorithm (see \cite{Kn}) provides a bijection between the set
\begin{displaymath}
\{(\mathbf{v},\mathbf{u}):\mathbf{v},\mathbf{u}\in \mathtt{N}_d, \mathbf{u}\text{ is }\mathbf{v}\text{-antidominant}\} 
\end{displaymath}
(which indexes indecomposable projective functors) and the set of all pairs $(\alpha,\beta)$ of semistandard 
Young tableaux of the same shape and content from $\{1,2,\dots,n\}$. For an indecomposable $1$-morphism $\theta$ in
$\cD_{n,r}$ we denote by $(\alpha_{\theta},\beta_{\theta})$ the corresponding pair of semistandard Young tableaux.

\begin{theorem}\label{thm423}
Let $\theta$ and $\theta'$ be two indecomposable $1$-morphisms in $\cD_{n,r}$.
\begin{enumerate}[$($i$)$]
\item\label{thm423.1} We have $\theta\sim_{\mathcal{L}}\theta'$ if and only if $\alpha_{\theta}=\alpha_{\theta'}$.
\item\label{thm423.2} We have $\theta\sim_{\mathcal{R}}\theta'$ if and only if $\beta_{\theta}=\beta_{\theta'}$.
\item\label{thm423.3} All two-sided cells in $\cD_{n,r}$ are strongly regular and satisfy \eqref{eq2}.
\end{enumerate}
\end{theorem}

\begin{proof}
All claims reduce to the corresponding statements for the regular block of $\mathcal{O}$, see \cite[7.1]{MM}, by 
first translating out of the source wall of the projective functor and then back onto the target wall.
\end{proof}

\subsection{Image completion}\label{s6.2}

Let $\cC$ be a fiat category and $\mathbf{M}\in \cC\text{-}\mathrm{mod}$. Define a new $2$-category
$\cA$ as follows: objects of $\cA$ are the same as objects of $\cC$; for $\mathtt{i},\mathtt{j}\in\cA$,
$1$-morphisms in $\cA(\mathtt{i},\mathtt{j})$ are all functors from $\mathbf{M}(\mathtt{i})$ to
$\mathbf{M}(\mathtt{j})$, isomorphic to $\mathbf{M}(\mathrm{F})$, where 
$\mathrm{F}\in \cC(\mathtt{i},\mathtt{j})$; $2$-morphisms in
$\cA$ are all natural transformations of functors; the composition in $\cA$ is given by the usual
composition of functors. The $2$-category $\cA$ will be called the {\em completion} of 
$\mathbf{M}(\cC)$.

\begin{proposition}\label{prop77}
The $2$-category $\cA$ is fiat.
\end{proposition}

\begin{proof}
Since $\mathbf{M}\in \cC\text{-}\mathrm{mod}$ and each $\mathbf{M}(\mathrm{F})$ is exact, 
each $\mathbf{M}(\mathrm{F})$ decomposes into a finite number of indecomposable functors. 
This implies that $\cA$ has finitely many indecomposable
$1$-morphisms (up to isomorphism). Clearly, all spaces of $2$-morphisms are finite dimensional
over $\Bbbk$. Now the claim follows from the standard fact (see e.g. \cite[1.2(c)]{BG}) 
that if some functor $\mathrm{X}:\mathbf{M}(\mathtt{i})\to \mathbf{M}(\mathtt{j})$ has a biadjoint 
$\mathrm{Y}:\mathbf{M}(\mathtt{j})\to \mathbf{M}(\mathtt{i})$, then every direct summand of
$\mathrm{X}$ has a biadjoint which is a direct summand of $\mathrm{Y}$.
\end{proof}

\begin{remark}\label{rem78}
{\rm
There is no obvious relation between $\cA$ and $\cC$. For instance, $\cA$ can have many more
$2$-morphisms, in particular, new idempotent $2$-morphisms, and thus many more indecomposable
$1$-morphisms. The other extreme is that we can map different $1$-morphisms to the same thing
making $2$-morphisms composable in $\cA$ which were not composable in $\cC$.   
}
\end{remark}

\vspace{1cm}

\noindent
Volodymyr Mazorchuk, Department of Mathematics, Uppsala University,
Box 480, 751 06, Uppsala, SWEDEN, {\tt mazor\symbol{64}math.uu.se};
http://www.math.uu.se/$\tilde{\hspace{1mm}}$mazor/.
\vspace{0.1cm}

\noindent
Vanessa Miemietz, School of Mathematics, University of East Anglia,
Norwich, UK, NR4 7TJ, {\tt v.miemietz\symbol{64}uea.ac.uk};
http://www.uea.ac.uk/$\tilde{\hspace{1mm}}$byr09xgu/.

\end{document}